\documentclass[11pt,a4paper,oneside]{article}
\usepackage{amsmath}
\usepackage{amssymb}
\usepackage{amscd}
\usepackage{amsfonts}
\usepackage[margin=1in]{geometry}
\usepackage[utopia]{mathdesign}
\usepackage[mathscr]{euscript}
\usepackage[utf8]{inputenc}
\usepackage[amsmath,amsthm,thmmarks]{ntheorem}
\newcommand{\qedhere}{}

\usepackage{multirow}
\usepackage{enumitem}
\usepackage{verbatim}
\usepackage{tikz}
\usetikzlibrary{matrix,arrows,calc,positioning}
\usepackage{caption}
\tikzset{>=stealth}

\usepackage{hyperref}

\hypersetup{%
  breaklinks,%
  ocgcolorlinks,
  colorlinks=true,%
  linkcolor=[rgb]{0.45,0.0,0.0},%
  citecolor=[rgb]{0,0,0.45}
}

\newcommand{\mc}{\mathscr}
\newcommand{\mb}{\mathbb}
\newcommand{\ralg}{\mathbb{R}_{\rm alg}}
\newcommand{\n}{${}$\newline}
\newcommand{\figsz}{\footnotesize}
\newcommand{\eps}{\varepsilon}

\renewcommand{\vec}{\mathbf}

\newtheorem{thrm}{Theorem}[section]
\newtheorem{lem}[thrm]{Lemma}

\newtheorem{prop}[thrm]{Proposition}

\newcommand{\df}[1]{\textbf{\textit{\color{cyan!10!black} #1}}}

\DeclareMathOperator{\conv}{conv}
\DeclareMathOperator{\cone}{cone}
\DeclareMathOperator{\ncone}{ncone}

\DeclareMathOperator{\face}{face}

\DeclareMathOperator{\nfan}{nfan}
\DeclareMathOperator{\linspan}{span}

\DeclareMathOperator{\vis}{vis}

\DeclareMathOperator{\projcl}{cl_\mb{P}}

\DeclareMathOperator{\proj}{proj}

\DeclareMathOperator{\trans}{trans}

\DeclareMathOperator{\side}{side}

\DeclareMathOperator{\labl}{indx}

\DeclareMathOperator{\bal}{balance}

\DeclareMathOperator{\trmr}{trans}
\DeclareMathOperator{\tent}{tent}
\DeclareMathOperator{\pris}{pris}
\DeclareMathOperator{\pyr}{pyr}
\DeclareMathOperator{\conn}{conn}
\DeclareMathOperator{\lamp}{lamp}
\DeclareMathOperator{\adapt}{adapt}

\DeclareMathOperator{\anch}{anch}

\newcommand*{\projeq}{%
  \mathrel{\vcenter{\offinterlineskip
  \hbox{{\tiny{\rm proj}}}\vskip+1pt\hbox{\scalebox{1.5}[1]{$\sim$}}}}}

\newcommand*{\combeq}{%
  \mathrel{\vcenter{\offinterlineskip
  \hbox{{\tiny{\rm comb}}}\vskip+1pt\hbox{\scalebox{2}[1]{\hskip-0pt$\sim$}}}}}

\newcommand*{\catprod}{\times}

\newcommand{\cv}{\mathrel{\hbox{\ooalign{$\cup$\cr\hidewidth\hbox{$\cdot\mkern5mu$}\cr}}}}
\newcommand{\Cv}{\mathrel{\hbox{\ooalign{$\bigcup$\cr\hidewidth\hbox{$\cdot\mkern9mu$}\cr}}}}

\newcommand{\nin}{\not\in }

\newcommand{\pentagon}{
\begin{tikzpicture}
\draw (90:3pt) -- (162:3pt) -- (234:3pt) -- (306:3pt) -- (378:3pt) -- cycle;
\end{tikzpicture}{} }

\newcommand{\bfpentagon}{ 
\begin{tikzpicture}
\draw[very thick] (90:3pt) -- (162:3pt) -- (234:3pt) -- (306:3pt) -- (378:3pt) -- cycle;
\end{tikzpicture} }

\newcommand{\subpentagon}{
\begin{tikzpicture}
\draw (90:2pt) -- (162:2pt) -- (234:2pt) -- (306:2pt) -- (378:2pt) -- cycle;
\end{tikzpicture} }

\newcommand{\cube}{
{\mkern0.5mu \begin{tikzpicture}[scale=0.1, baseline= -2pt]
\draw[gray] (-1,-1) -- (0,0) -- (2,0) 
            (0,0) -- (0,2);
\draw (1,1) -- (2,2) -- (2,0) -- (1,-1) -- (-1,-1) -- (-1,1) -- (0,2) -- (2,2)
      (1,-1) -- (1,1) -- (-1,1);
\end{tikzpicture} \mkern0.5mu}{} }

\newcommand{\subcube}{
{\mkern2mu \begin{tikzpicture}[scale=0.06, baseline= -1pt]
\draw[gray] (-1,-1) -- (0,0) -- (2,0) 
            (0,0) -- (0,2);
\draw (1,1) -- (2,2) -- (2,0) -- (1,-1) -- (-1,-1) -- (-1,1) -- (0,2) -- (2,2)
      (1,-1) -- (1,1) -- (-1,1);
\end{tikzpicture} \mkern2mu} }

\begin{document}

\title{%
Antiprismless, \\
or: Reducing Combinatorial Equivalence to Projective Equivalence in Realizability Problems for Polytopes
}

\author{Michael Gene Dobbins 
\thanks{Department of Mathematical Sciences, Binghamton University (SUNY),
Binghamton, NY, USA. \newline
\texttt{michaelgenedobbins@gmail.com} \newline
This research was supported by NRF grant 2011-0030044 (SRC-GAIA) funded by the government of Korea.\newline
}}

\date{ }








\maketitle


\begin{abstract}
This article exhibits a 4-dimensional combinatorial polytope that has no antiprism, answering a question posed by Bernt Lindst\"om.  
As a consequence, any realization of this combinatorial polytope has a face that it cannot rest upon without toppling over. 
To this end, we provide a general method for solving a broad class of realizability problems.  Specifically, we show that for any semialgebraic property that faces inherit, the given property holds for some realization of every combinatorial polytope if and only if the property holds from some projective copy of every polytope. 
The proof uses the following result by Below.  Given any polytope with vertices having algebraic coordinates, there is a combinatorial ``stamp'' polytope with a specified face that is projectively equivalent to the given polytope in all realizations.  
Here we construct a new stamp polytope that is closely related to Richter-Gebert's proof of universality for 4-dimensional polytopes, and we generalize several tools from that proof.
\end{abstract}

\section{Introduction}

The combinatorial type of a polytope is defined by the partial ordering of its face lattice. 
We generally ``see'' a partial ordering by drawing its Hasse diagram;  
see Figure \ref{fig:trifl}.
If we draw the Hasse diagram of the face lattice of a polytope, 
we may observe that it resembles the 1-skeleton of a larger polytope.  
For example, the Hasse diagram of a simplex's face lattice is the 1-skeleton of a hypercube standing on a vertex. 
The 1-skeleton alone does not uniquely determine the combinatorial type of a polytope, but there is a natural extension of the Hesse diagram that does, the intervals of a poset ordered by inclusion. 
When the original poset is a combinatorial polytope, 
the resulting poset of intervals shares some basic properties with combinatorial polytopes, such as being a lattice and satisfying Euler's formula \cite{lindstrom1971realization}.

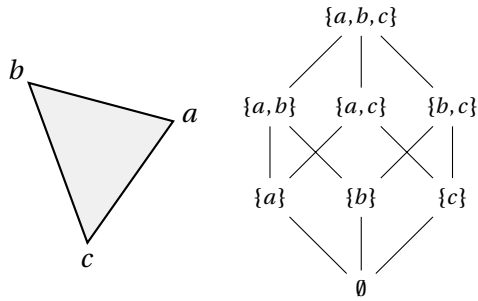
\begin{figure}[!h]

\begin{center}
\begin{tikzpicture}[scale=1.2]

 \path (-2,1.5) coordinate (t);
 \path (t) +(20:1) coordinate (a);
 \path (t) +(130:1) coordinate (b);
 \path (t) +(270:1) coordinate (c);
 
\filldraw[thick,fill=black!6] (a) -- (c) -- (b) -- cycle;

 \path (t) +(20:1.2) coordinate (a);
 \path (t) +(130:1.2) coordinate (b);
 \path (t) +(270:1.2) coordinate (c);

\node at (a) {$a$}; 
\node at (b) {$b$}; 
\node at (c) {$c$};

 \node at (1,3) (abc) {\scalebox{.8}{$\{a,b,c\}$}};
 \node at (0,2) (ab) {\scalebox{.8}{$\{a,b\}$}};
 \node at (1,2) (ac) {\scalebox{.8}{$\{a,c\}$}};
 \node at (2,2) (bc) {\scalebox{.8}{$\{b,c\}$}};
 \node at (0,1) (a) {\scalebox{.8}{$\{a\}$}};
 \node at (1,1) (b) {\scalebox{.8}{$\{b\}$}};
 \node at (2,1) (c) {\scalebox{.8}{$\{c\}$}};
 \node at (1,0) (o) {\scalebox{.8}{$\emptyset$}};
 
 \draw (o) -- (a) -- (ac) -- (c) -- (o)
 (b) -- (ab) -- (abc) -- (bc) -- (b)
 (ac) -- (abc)
 (a) -- (ab)
 (c) -- (bc)
 (o) -- (b);  

\end{tikzpicture}

\caption{ A triangle and a Hasse diagram of its face lattice.}\label{fig:trifl}
\end{center}

\end{figure}

In 1971 Lindstr\"om asked whether the intervals of a polytope's face lattice always form a new combinatorial polytope \cite{lindstrom1971problem}.  
In this article, we will see that this is not the case.  Moreover, we will construct a 4-polytope such that the poset of intervals of its face lattice is not the combinatorial type of any polytope. 
An equivalent question appears in Gr\"unbaum's text book \cite{grunbaum2003convex} and has applications in linear optimization \cite{broadie1985antiprisms}. 
Namely, does every polytope have an antiprism?  An antiprism is the combinatorial dual of the interval polytope (see Figure~\ref{fig:cubeantiprism}). 
Anders Bj\"orner announced the answer in 3 dimensions; every 3-polytope does have an antiprism \cite{bjorner1997antiprism}, but this result remains unpublished. 

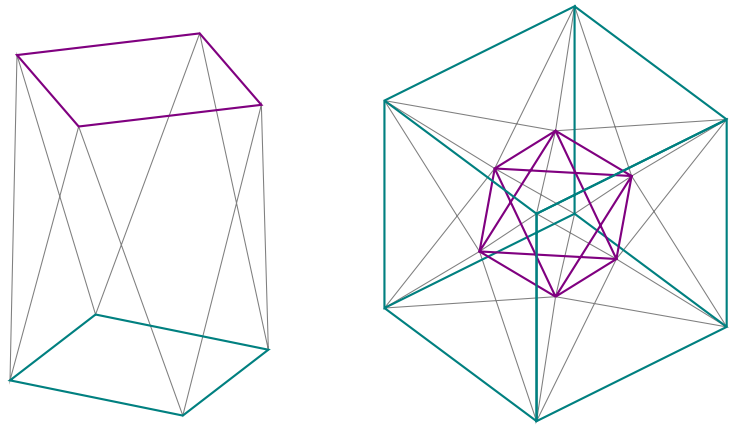
\begin{figure}[ht]
\centering

\begin{tikzpicture}

\matrix[column sep=1.5cm]
{

\begin{scope}[scale=1.4,x={(1cm,.3cm)},z={(0cm,2.7cm)},y={(.8cm,-.4cm)}]

\path
(0,0,-.2)
+(20:1) coordinate (a) 
+(110:1) coordinate (b) 
+(20:-1) coordinate (c) 
+(110:-1) coordinate (d) 
(0,0,.8)
+(65:1) coordinate (e) 
+(155:1) coordinate (f) 
+(65:-1) coordinate (g) 
+(155:-1) coordinate (h) 
;

\draw[black!50]
(a) -- (e) -- (b) -- (f) -- (c) -- (g) -- (d) -- (h) -- cycle
;

\draw[thick,teal]
(a) -- (b) -- (c) -- (d) -- cycle
;
\draw[thick,violet]
(e) -- (f) -- (g) -- (h) -- cycle
;

\end{scope}

&

\begin{scope}[scale=2.5,x={(1cm,.5cm)},y={(0cm,1.1cm)},z={(.8cm,-.6cm)}]

\path
(0,0,0) coordinate (a) -- (1,0,0) coordinate (b) -- (1,1,0) coordinate (c) -- (0,1,0) coordinate (d) -- cycle
(0,0,1) coordinate (e) -- (1,0,1) coordinate (f) -- (1,1,1) coordinate (g) -- (0,1,1) coordinate (h) -- cycle
;
\path
(.5,.1,.5) coordinate (i) ++(0,.8,.0) coordinate (j) 
(.1,.5,.5) coordinate (k) ++(.8,0,0) coordinate (l) 
(.5,.5,.1) coordinate (m) ++(0,0,.8) coordinate (n) 
;

\draw[black!50]
(i) -- (a) -- (m) -- (b) -- (i)
(i) -- (e) -- (n) -- (f) -- (i)
(j) -- (c) -- (l) -- (g) -- (j)
(j) -- (d) -- (k) -- (h) -- (j)
(a) -- (k) -- (e)
(b) -- (l) -- (f)
(c) -- (m) -- (d)
(g) -- (n) -- (h)
;
\draw[thick,teal]
(a) -- (b) -- (c) -- (d) -- cycle
(e) -- (f) -- (g) -- (h) -- cycle
(a) -- (e)
(b) -- (f)
(c) -- (g)
;
\draw[thick,violet]
(i) -- (k) -- (j) -- (l) -- (i)
(i) -- (m) -- (j) -- (n) -- (i)
(m) -- (k) -- (n) -- (l) -- (m)
;
\draw[thick,teal]
(d) -- (h)
(g) -- (h) -- (e)
;

\end{scope}

\\
};

\end{tikzpicture}
\caption{\textsc{Left:} The antiprism of a square. \textsc{Right:} The antiprism of a cube.}\label{fig:cubeantiprism}
\end{figure}

Broadie gave sufficient conditions for a polytope to have an antiprism \cite{broadie1985antiprisms}. 
These conditions ask for a perfectly centered realization of the original polytope \cite{fukudaweibel2007fvectors}.  
Perfectly centered has a nice physical interpretation; it says a polytope with a specified center of mass 
can rest on any face without toppling over; we may consider a polytope filled with some inhomogeneous material. 
Since every 3-polytope has an antiprism, every combinatorial 3-polytope has a realization that can rest on each of its faces. 
Here, we will give both necessary and sufficient conditions for the realizability of a polytope's antiprism,  
then construct a 4-polytope without an antiprism.  Consequently, every realization of this 4-polytope has some face on which it cannot rest.  Note, however, that the face on which it cannot rest is not necessarily a facet. 

As part of this construction, we provide a technique for answering questions of the following form. 
Does a certain geometric property hold for some realization of every combinatorial polytope? 
Such questions are made difficult by the universality theorem for polytopes.  The universality theorem states that for any primary basic semialgebraic set $X$, there exists a poset with realization space (modulo isometries) that is homotopy equivalent to $X$.  
J{\"u}rgen Richter-Gebert showed that universality holds even for polytopes of dimension 4 \cite{richter1996realization}. 
As a consequence, 
searching for a realization of a certain combinatorial polytope that satisfies a certain geometric property, can be as hard as searching for a point in a semialgebraic set.  This may be difficult, since a semialgebraic set may be disconnected or have holes or other unwanted features for a search space. 

For a broad class of properties of polytopes, 
we show that the problem of determining whether such a property 
holds for some realization of every combinatorial type can be reduced to determining whether the property holds for some realization of every projective type.  
This is a considerable improvement since, in contrast to the realization spaces of polytopes with fixed combinatorial type, the space of polytopes (up to isometry) with fixed projective type is convex.  
We will also see that, 
when such a realization does not always exist, 
there is a gap of at most 2 between the lowest dimension where this fails for combinatorial types and the lowest dimension where it fails for projective types. 

For a geometric property that is general enough to be relevant in any dimension, if the property holds for a polytope, then in many cases, it holds for the polytope's faces as well.  
We say faces projectively inherit a property when, for any face of a polytope with this property, some projective copy of the face has the property.  
Theorem \ref{thrm:log_comb=proj} gives a reduction from combinatorial type to projective type for semialgebraic properties that faces projectively inherit, and Theorem \ref{thrm:alg_comb=proj} gives a reduction from combinatorial type to projective type for any propery faces projectively inherit provided we restrict ourselves to polytopes with vertices having algebraic coordinates. 

An example of such a property is, ``The polytope's vertices have rational coordinates''.  Trivially, if this is true of a polytope then it is also true of its faces. 
This reduces the question ``Can every combinatorial 4-polytope be realized with rational coordinates?'' to the question 
``Does every polygon with algebraic vertices have a projective copy with rational vertices?''.  Answering the first question was a considerable hurdle that paved the way for Gale duality and the universality theorem \cite{ziegler2008nonrational}.  The answer to the second question, however, is easily seen to be no.  Just consider the regular pentagon. 

\begin{figure}[ht]
\centering

\begin{tikzpicture}[scale=1.2]

\draw[thick]
(90:1) coordinate (a0)
\foreach \i in {1,2,3,4}
 { -- ({90+72*\i}:1) coordinate (a\i) }
-- cycle
;

\draw
(a2) -- (a0) -- (a3)
(a1) -- (a4)
;

\node[shift={(-.1,.1)},anchor=base] at (a1) {\figsz $a$};
\node[shift={(.8,.1)}, anchor=base] at (a1) {\figsz $b$};
\node[shift={(-.8,.1)},anchor=base] at (a4) {\figsz $c$};
\node[shift={(.1,.1)}, anchor=base] at (a4) {\figsz $d$};

\end{tikzpicture}
\caption{For any projective copy of a regular pentagon, the cross-ratio $(a,b\mathbin{|}c,d) = \frac{1+\sqrt{5}}{2}$ is irrational, so the vertices cannot have rational coordinates.}
\end{figure}
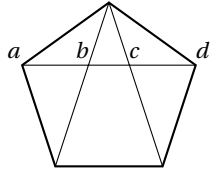


This reduction is a consequence of a construction we call a stamp of a polytope. 
In $\mathbb{R}^3$ it is known that faces of polyhedra are prescribable \cite{barnette1970preassigning}. That is, given a realization of a face of a combinatorial 3-polytope, it is always possible to extend this to a realization of the entire polytope. This does not hold in higher dimensions 
\cite{kleinschmidt1976facets}\cite{ziegler2007lecture}, and
a stamp gives the strongest possible violation of this for polytopes in general.  
A stamp is a combinatorial polytope that forces a specified face to have a fixed projective type in all realizations.  
In his unpublished thesis, Below constructed a stamp for any projective type of polytope 
having vertices with coordinates in the real algebraic completion of the rationals
\cite{below2002complexity}. 
This article will give a different stamp construction. 

The stamp construction presented here has the advantage that it closely follows Richter-Gebert's proof of universality for 4-polytopes \cite{richter1996realization}. 
The proof of the universality theorem encodes a system of polynomial constraints into a poset in such a way that a realization of the poset by a polytope corresponds to a solution of the given system.  
The stamp construction uses similar techniques to encode a system of polynomial constraints into a polytope and then force the coordinates of certain vertices to satisfy those polynomial constraints. 
While both proofs involve many technical details, the reader who is already familiar with Richter-Gebert's proof may easily recognize the modifications of that proof in the stamp construction and find it much more accessible.
Several lemmas are also of independent interest as generalizations of those in \cite{richter1996realization}.

{

\paragraph{Organization of the paper.}

Section \ref{algred} deals with the reduction from combinatorial equivalence to projective equivalence. 
Section \ref{antiprism} proves some basic results about prismoids, then shows that the existence of a balanced pair is equivalent to the existence of an antiprism, and finally presents a polytope without an antiprism.  Section \ref{stamp} constructs the stamp of a polytope. 
Sections \ref{algred} and \ref{antiprism} depend on Section \ref{stamp}, 
but Section \ref{stamp} is placed later in the text to spare the reader the extensive details of the construction of stamp polytopes until motivated by their use.   
Finally, Section~\ref{question} leaves the reader with some open questions.

\paragraph{Terminology and notation.}

In this article, we assume that a polytope's faces are indexed by a poset with the order of indices in the poset indicating containment of faces. 
We denote the face of a polytope $P$ that has index $f$ by $F = \face(P,f)$, and we may refer to $F$ as the face $f$ of $P$.  
We say polytopes are combinatorially equivalent (or have the same combinatorial type) when their faces are indexed by the same post, 
and we will makes use of the implied correspondence between their faces.   Note that a polytope may have non-trivial symmetry, and some properties considered depend on the indexing of the faces in a way that is not preserved by these symmetries.  

An invertible projective transformation on $\mb{R}^n$ is called a projectivity, and 
we say polytopes $P$ and $Q$ are projectively equivalent (or have the same projective type) when they are combinatorially equivalent their exists a projectivity $\pi$ such that for each index $f$, $\pi(\face(P,f)) = \face(Q,f)$.  Note that the restriction of $\varphi$ to $P$ is bounded, but does not necessarily preserve orientation.  We may also say $Q$ is a projective copy of $P$.

We denote column vectors by $[x_1; x_2; \dots; x_n]$, and we may include a set among the entries of a vector to indicate a Cartesian product.  For example, $[P;1] = \{[x_1;\dots;x_n;1] : [x_1;\dots;x_n] \in P \}$. 
A brief glossary of notation follows.  

\newlength{\myitemwidth}
\setlength{\myitemwidth}{1\textwidth}
\addtolength{\myitemwidth}{-1.5cm}
\begin{tabular}{l@{\, :\ }p{\myitemwidth}} 
$\ralg$ & The real algebraic closure of the rationals. \\
$f\wedge g$ & The meet of elements of a lattice; e.g.\ logical conjunction ``and''. \\
$f \vee g$ & The join of elements of a lattice; e.g.\ logical disjunction ``or''. \\
$\bot$ ($\top$) & The least (greatest) element of a bounded poset. \\ 
$\mc{P}^*$ & The poset $\mc{P}$ with order reversed. \\
$\mc{P} \catprod \mc{Q}$ & The categorical product of posets, 
$(a,x) \leq (b,y)$ when $a \leq b$ and $x \leq y$. \\
$\mc{P}_1 \wedge \mc{P}_2$ & The common refinement of a pair of sublattices. \\
$\labl(P)$ & The poset indexing the face lattice of a polytope $P$. \\
$\face(P,f)$ & The face labeled by $f \in \mc{P}$ of a polytope $P$ realizing the poset $\mc{P}$.  \\
$P^\circ$ & The relative interior of a set $P$. \\
$P^*$ & The polar of a centered polytope (or cone) $P$. \\
$P_1 \cv P_2$ & The convex join, $\conv(P_1 \cup P_2)$. \\ 
$\cone(P,f)$ & The cone over the face $f$ of $P$. \\
$\ncone(P,f)$ & The normal cone of the face $f$ of $P$. \\
$\nfan(P)$ & The normal fan of $P$. \\
$P \combeq Q$ & $P$ and $Q$ are combinatorially equivalent. \\
$[P]_\text{comb}$ & The class of polytopes that are combinatorially equivalent to $P$. \\
$P \projeq Q$ & $P$ and $Q$ are projectively equivalent. \\
$[P]_\text{proj}$ & The class of polytopes that are projectively equivalent to $P$. \\
\end{tabular}

}


\section{Combinatorial and Projective Equivalence} 
\label{algred}

Later we will see how to construct the stamp of an algebraic polytope, but in this section we consider the consequences of its existence.  For now, a \df{stamp} is the pair $(\mc{S}_P,f_P)$ implied by the following theorem.

\begin{thrm}\label{the:paf} 
Given an algebraic $d$-polytope $P$, 
there exists a combinatorial ${(d\! +\! 2)}$-polytope $\mc{S}_P$ with a specified face $f_P \in \mc{S}_P$ 
such that for any realization $S$ of $\mc{S}_P$, the specified face is projectively equivalent to the given polytope, $\face(S,f_P) \projeq P$.
\end{thrm}

Such a polytope has been constructed in \cite[p.\ 134]{below2002complexity}. 
The stamp helps us answer questions about properties that faces inherit, or more generally the following class of predicates. 
Let $\psi$ be a predicate of several algebraic polytopes of the same combinatorial type. 
We say the face $f \in \mc{P}$ \df{projectively inherits} $\psi$ when, 
if $P_1,\dots,P_n$ are realizations of $\mc{P}$ such that $\psi(P_1,\dots,P_n)$ is true, 
then there are projectivities 
$\pi_1,\dots,\pi_n$ such that $\psi(\pi_1(\face(P_1,f)),\dots,\pi_n(\face(P_n,f)))$ is true. 
Recall that a ridge of a polytope is a face of co-dimension 2. 

\begin{figure}[ht]
\centering

\begin{tikzpicture}[scale=1.5]

\draw[thick] (0,0) coordinate (a) -- ++(0:1) coordinate (b) -- ++(40:0.6) coordinate (c);
\draw (c) -- ++(0:-1) coordinate (d) -- (a);
\draw[thick]
(0,1) coordinate (e) -- ++(0:1) coordinate (f) -- ++(40:0.6) coordinate (g) -- ++(0:-1) coordinate (h) -- cycle
(a) -- (e)
(b) -- (f)
(c) -- (g);
\draw
(d) -- (h);

\begin{scope}[xshift=2cm]
\draw[thick] (0,0) coordinate (a) -- ++(0:1) coordinate (b) -- ++(40:0.6) coordinate (c);
\draw (c) -- ++(0:-1) coordinate (d) -- (a);
\draw[thick]
(0,2) ++(0:.5) ++(40:.3) coordinate (x)
($(a)!.5!(x)$) coordinate (e) -- ($(b)!.5!(x)$) coordinate (f) 
-- ($(c)!.5!(x)$) coordinate (g) -- ($(d)!.5!(x)$) coordinate (h) -- cycle
(a) -- (e)
(b) -- (f)
(c) -- (g);
\draw
(d) -- (h);

\draw[dotted]
(x) -- (e)
(x) -- (f)
(x) -- (g)
(x) -- (h);
\end{scope}

\begin{scope}[xshift=4cm]
\draw[thick] (0,0) coordinate (a) -- ++(0:1) coordinate (b) -- ++(40:0.6) coordinate (c);
\draw (c) -- ++(0:-1) coordinate (d) -- (a);
\draw[thick]
(0,1) coordinate (e) -- ++(0:.5) coordinate (f)
 -- ++(40:0.6) coordinate (g) -- ++(0:-.5) coordinate (h) -- cycle
(a) -- (e)
(b) -- (f)
(c) -- (g);
\draw
(d) -- (h);

\draw[dotted]
(e) -- (0,2) -- (f)
(0,2) ++(40:.6) coordinate (x)
(g) -- (x) -- (h)
(0,2) -- (x);
\end{scope}

\end{tikzpicture}

\caption{Three combinatorially equivalent polytopes.  Only the left two are projectively equivalent.}
\end{figure}
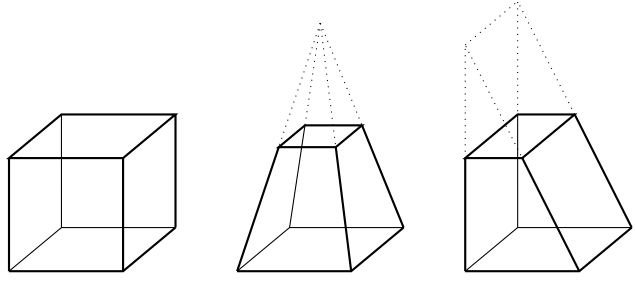

\vbox{
\begin{thrm}\label{thrm:alg_comb=proj} 
Let $\psi$ be 
a predicate of several algebraic polytopes of the same combinatorial type that ridges projectively inherit.  Then, $\psi$ holds for some realization of every combinatorial type of polytope if and only if it holds for some realization of every algebraic projective type.
Moreover, there can be a gap of at most 2 between the lowest dimension of a combinatorial type where $\psi$ always fails and the lowest dimension of an algebraic projective type where $\psi$ always fails, 
\[\begin{array}{rl@{\: }l}
& \forall p_i \in \ralg^{d+2} \exists P_j \in [\Cv_{i=1}^m p_i]_{\rm comb}. & \psi(P_1,\dots,P_n) \\ 
\Rightarrow &
\forall p_i \in \ralg^{d\phantom{+2}} \exists P_j \in [\Cv_{i=1}^m p_i]_{\rm{proj}}. & \psi(P_1,\dots,P_n) \\
\Rightarrow &
\forall p_i \in \ralg^{d\phantom{+2}} \exists P_j \in [\Cv_{i=1}^m p_i]_{\rm comb}. & \psi(P_1,\dots,P_n). 
\end{array}\]
\end{thrm}
}

\begin{proof}
Since projective equivalence is finer than combinatorial equivalence, we have the `if' direction trivially.  For the other direction, let $\psi$ be a predicate that ridges projectively inherit and suppose $\psi$ holds for some realizations of every combinatorial polytope.  Consider an algebraic polytope $P$.  Since the stamp $\mc{S}_P$ has realizations where $\psi$ holds, it must also hold for some projective copies of the face $f_P$ of each of these realizations, and these faces are all projectively equivalent to $P$ by Theorem \ref{the:paf}.  Thus, $\psi$ holds for some projective copies of $P$.  Also, since $\mc{S}_P$ is 2 dimensions higher than $P$, if $\psi$ holds for some realizations of every combinatorial polytope up to dimension $d+2$, then it holds for some projective copies of every algebraic polytope up to dimension $d$. 
\end{proof}

Generally it is common to consider polytopes in $\mb{R}^d$, so it would be nicer if Theorem~\ref{thrm:alg_comb=proj} were not restricted to polytopes in $\ralg^d$.  If we simply removed this condition, the resulting claim would be false.  
Instead, we can replace this restriction on the space of polytopes to further restrictions on the kind of properties considered.  Specifically, we require the $\psi$ to be semialgebraic.  
We say $\psi$ is \df{semialgebraic property} when, 
for each combinatorial type of polytope $\mc{P}$, the restriction of $\psi$ to realizations of $\mc{P}$ is expressible as a formula on the coordinates of the vertices in the language of real closed fields $({+},\,{\cdot}\,,\,0,\,1,\,{\leq})$. 
That is, for any combinatorial polytope $\mc{P}$ with $m$ vertices, there is some formula $\psi_\mc{P}$ in the language of real closed fields with free variables $v = (v_{1,1,1},\dots,v_{d,m,K})$ such that for any model $\mb{K}$ of real closed fields and any polytopes $P_1,\dots,P_K$ of type $\mc{P}$ with vertex coordinates $c \in \mb{K}^{dmK}$, 
$\psi(P_1,\dots,P_K)$ is true if and only if the formula $\psi_\mc{P}[c/v]$ defined by substituting each vertex coordinate $v_{i,j,k}$ by the constant $c_{i,j,k}$ is true in ${\mb{K}}$.  
For an accessible review of real closed fields and model theory see \cite[Section IV.23]{princeton2008}.

\begin{thrm}\label{thrm:log_comb=proj}
If $\psi$ is a semialgebraic property of several polytopes of the same combinatorial type such that ridges always projectively inherit $\psi$, then 
$\psi$ holds for some realizations of every combinatorial type of polytope if and only if it holds for some realizations of every projective type.
Moreover, there can be a gap of at most 2 dimensions.
\end{thrm}


\begin{proof}
Briefly, Theorem~\ref{thrm:log_comb=proj} follows from the fact that $\mb{R}$ and $\ralg$ are elementarily equivalent \cite{tarski1951decision}, and both realizability and projective equivalence are definable in the language of real closed fields. 

Note that we will exclusively use $\wedge$ for logical conjunction in this proof. 
For a formula $\phi$ with free variable $x$, and a formula $\theta$, recall $\phi[\theta/x]$ denotes the formula where $x$ is replaced by $\theta$.
We will also use conventional notation in $\mb{R}^d$ as abbreviations for the formulas that can easily be written in the language of real closed fields.  For example, $x-y=z$ should be understood as $x=z+y$.

First we write a formula $\psi'_\mc{P}$ that says there are realizations of $\mc{P}$ where $\psi_\mc{P}$ holds. 
For this, write a formula $\rho_\mc{P}$ with $dmK$ free variables $v_{i,j,k}$ for each coordinate of each vertex $\vec{v}_{j,k}$ of each polytope $P_k$ that says these are indeed vertex coordinates of a polytope of type $\mc{P}$.  Let $\mc{F}_1,\dots,\mc{F}_n$ be the set vertices in each facet of $\mc{P}$. 
\[ \rho_\mc{P} := \exists \vec{a}_{1,1} \dots \exists\vec{a}_{n,K}. \nu_\mc{P} \]
\[\nu_\mc{P} := 
\bigwedge_{ k=1 }^K 
\left( \bigwedge_{v_j\in \mc{F}_i}  \left<\vec{a}_{i,k},\vec{v}_{j,k}-\vec{\tau}_k\right> = 1 \right) \wedge 
\left( \bigwedge_{v_j\not\in \mc{F}_i}  \left<\vec{a}_{i,k},\vec{v}_{j,k}-\vec{\tau}_k\right> < 1 \right)  \]
\[ \vec{\tau}_k := \tfrac{1}{m}\sum_{j=1}^{m} \vec{v}_{j,k}  \]
Note that $\vec{\tau}_k$ is a formula for a translation vector that centers $P_k$.
The formula $\nu_\mc{P}$
includes free variables for the half-spaces supporting each facet of a centered translation of the polytope, and says that the vertices of a facet are on the boundary of its supporting half-space and the rest of the vertices are in the interior of this half-space.
Now write a formula $\psi'_\mc{P}$ asserting the existence of a realization where $\psi_\mc{P}$ holds.
\[ \psi'_\mc{P} := \exists\vec{v}_{1,1}  \dots  \exists\vec{v}_{m,K}.\: \rho_\mc{P} \wedge \psi_\mc{P} \ .\]

Next, write a formula $\chi_\mc{P}$ saying for $K$ polytopes with combinatorial type $\mc{P}$ that there is a projective copy where predicate $\psi_\mc{P}$ holds.  
For this we represent projectivities $\pi_k$ on $\mathbb{R}^d$ by $(d{+}1)\times(d{+}1)$ matrices 
$M_k$ acting on homogeneous coordinates.
\[ \chi_\mc{P} := \exists M_1 \dots \exists M_K  \exists x_{1,1} \dots \exists x_{m,K} . \mu_\mc{P}, \quad
M_k := \left[ \begin{array}{cc} A_k & b_k \\ c_k^* &  1 \end{array} \right]
\]
\[ \mu_\mc{P} := 
\bigwedge_{j,k = 1,1}^{m_\mc{P},K} \left( x_{j,k} {\cdot} \left(\left< \vec{c}_k,\vec{w}_{j,k} \right> +1\right) = 1 \right)
\wedge (\rho_\mc{P} \wedge \psi_\mc{P})[x_{j,k} {\cdot} (A_k \vec{w}_{j,k} +\vec{b}_k)/\vec{v}_{j,k}] \]
Note that the formula $\mu_\mc{P}$ is defined by 
replacing each coordinate of $\vec{v}_{j,k}$ in the formula $\rho_\mc{P} \wedge \psi_\mc{P}$ with a formula for the corresponding coordinate of $\pi_k(\vec{v}_{j,k}) := x_{j,k} \vec{u}_{1:d}$ where $\vec{u}=M_k[v_{j,k} ; 1]$ and $x_{j,k} = u_{d+1}^{-1}$. 
Now write a formula $\chi'_\mc{P}$ asserting that every realization has a projective copy where $\psi_\mc{P}$ holds.
\[ \chi'_\mc{P} := \forall \vec{v}_{1,1}  \dots \forall \vec{v}_{m_\mc{P},K}.\: \rho_\mc{P} \Rightarrow \chi_\mc{P} \ . \]

In both $\mb{R}$ and $\ralg$ we have immediately that the existence of a realization of every projective type where the predicate holds implies the existence of such a realization for every combinatorial type.
For the other direction, 
suppose there is some combinatorial $d$-polytope $\mc{P}$ with realizations in $\mathbb{R}^d$ such that $\psi$ does not hold for any projective copies in $\mathbb{R}^d$. That is, $\mathbb{R} \models \neg\chi'_\mc{P}$.  Then, $\ralg \models \neg\chi'_\mc{P}$, which asserts the existence of an algebraic polytope $P$ where $\psi$ does not hold for any algebraic projective copies $\pi_1(P),\dots,\pi_K(P)$.  Let $\mc{S}_P$ be the combinatorial $(d\!+\!2)$-polytope that is the stamp of $P$.  Then, $\ralg \models \neg\psi'_{\mc{S}_P}$, and therefore $\mathbb{R} \models \neg\psi'_{\mc{S}_P}$.  Thus, we have found a  combinatorial $(d\!+\!2)$-polytope such that $\psi$ does not hold for any realization in $\mathbb{R}^d$.
\end{proof}

\section{Antiprisms}
\label{antiprism}

A polytope is a \df{prismoid} when every vertex of the polytope is in one of two nonintersecting faces, which we call the \df{bases} of the prismoid.  
That is, every prismoid $P$ is of the form 
\[ {P} = B_0 \cv B_1 \ =  \{t_0 B_0 +t_1 B_1: t_i \geq 0, t_0+ t_1=1\} \]
where ${B_0, B_1 }$ are disjoint faces. 
The \df{sides} of the prismoid are faces that are not contained in either base, along with the trivial side $\bot$. 
When a combinatorial polytope is a prismoid, we call it a combinatorial prismoid. 
Some examples of prismoids are pyramids, tents, prisms, and antiprisms. 

We define a purely combinatorial construction, called an abstract prismoid. 
The definition is motivated by the fact that a face of a prismoid is determined by its intersection with each of the bases.  For bounded posets $\mc{B}_0$, $\mc{B}_1$, an \df{abstract prismoid} $\mc{P}$ with these bases is a bounded subposet of the categorical product 
$\mc{B}_0 \catprod \mc{B}_1$ 
such that the bases themselves are included as $(f_0, \bot)$, $(\bot, f_1) \in \mc{P}$ for all $f_i \in \mc{B}_i$. 
All faces that are not in a base and the face $\bot = (\bot,\bot)$ are sides of $\mc{P}$, denoted 
\[\side(\mc{P}) := \{ (f_0,f_1) \in \mc{P} : (f_0{=}\bot) \Leftrightarrow (f_1{=}\bot) \} .\]  

\vbox{
\begin{lem}\label{lem:prism} 
Every combinatorial prismoid is isomorphic to an abstract prismoid. 
And, an abstract prismoid $\mc{P} \subset \mc{B}_0 \catprod \mc{B}_1$ can be realized if and only if
there are realizations $B_i$ of the bases $\mc{B}_i$ such that the combinatorial type of the common refinement of the normal fans of the $B_i$ is the dual of the sides of $\mc{P}$,  
\[ \labl({\rm nfan}(B_0) \wedge {\rm nfan}(B_1))^* = \side(\mc{P}).  \]
Moreover, 
the set of all realizations of $\mc{P}$ is the set of all polytopes that are projectively equivalent to $[B_0 ; 0] \cv [B_1 ; 1]$ for some $B_i$ satisfying the above.
\end{lem}
}


\begin{figure}[h]
\centering

\begin{tikzpicture}
\begin{scope}[scale=1/2]
\path (0,-2.25)
 +(75:3 and 1.2) coordinate (a)
 +(165:3 and 1.2) coordinate (b)
 +(255:3 and 1.2) coordinate (c)
 +(345:3 and 1.2) coordinate (d)
 (0,2.25)
 +(120:2.5 and 1) coordinate (e)
 +(240:2.5 and 1) coordinate (f)
 +(360:2.5 and 1) coordinate (g)
 ($(g)!.5!(a)$) coordinate (h)
 ($(e)!.5!(a)$) coordinate (i)
 ($(e)!.5!(b)$) coordinate (j)
 ($(f)!.5!(b)$) coordinate (k)
 ($(f)!.5!(c)$) coordinate (l)
 ($(g)!.5!(d)$) coordinate (m);
\end{scope}

\draw[thick]
 (e) -- (f) -- (g) -- cycle
 (e) -- (b) -- (f) -- (c) -- (d) -- (g)
 (b) -- (c) -- (d)
;
\draw
 (d) -- (a) -- (b) 
 (g) -- (a) -- (e) 
;
\draw[gray] 
(h) -- (i) -- (j) -- (k) -- (l) -- (m) -- cycle
;

\end{tikzpicture}

\medskip

\begin{tikzpicture}[thick]
\matrix[column sep=10]
{
\begin{scope}[scale=2/3]
\path (0,0) coordinate (o)
 (45:2) coordinate (a)
 (135:2) coordinate (b)
 (225:2) coordinate (c)
 (315:2) coordinate (d)
 (90:1.5) coordinate (e)
 (210:1.5) coordinate (f)
 (330:1.5) coordinate (g)
 ($(g)!.5!(a)$) coordinate (h)
 ($(e)!.5!(a)$) coordinate (i)
 ($(e)!.5!(b)$) coordinate (j)
 ($(f)!.5!(b)$) coordinate (k)
 ($(f)!.5!(c)$) coordinate (l)
 ($(g)!.5!(d)$) coordinate (m)
 (0:2.5) coordinate (ab)
 (90:2.5) coordinate (bc)
 (180:2.5) coordinate (cd)
 (270:2.5) coordinate (da)
 (270:2) coordinate (fg)
 (30:2) coordinate (ge)
 (150:2) coordinate (ef);
\end{scope}
 
 \draw (a) -- (b) -- (c) -- (d) -- cycle; 
 \draw[teal,->] (o) -- (ab);
 \draw[teal,->] (o) -- (bc);
 \draw[teal,->] (o) -- (cd);
 \draw[teal,->] (o) -- (da);
&
 
 \draw (h) -- (i) -- (j) -- (k) -- (l) -- (m) -- cycle; 
 \draw[teal,->] (o) -- (ab);
 \draw[teal,->] (o) -- (bc);
 \draw[teal,->] (o) -- (cd);
 \draw[teal,->] (fg) -- (da);
 \draw[violet,->] (o) -- (fg);
 \draw[violet,->] (o) -- (ge);
 \draw[violet,->] (o) -- (ef); 
&
 
 \draw[white] (ab) -- (cd);
 \draw (e) -- (f) -- (g) -- cycle; 
 \draw[violet,->] (o) -- (fg);
 \draw[violet,->] (o) -- (ge);
 \draw[violet,->] (o) -- (ef); 
\\
};

\end{tikzpicture}

\caption{
\textsc{Top:} A prismoid with a triangular and square base.
\textsc{Bottom:} The prismoid bases and a horizontal slice with the common refinement of the normal fans of the bases.
}\label{fig:prismoid}
\end{figure}
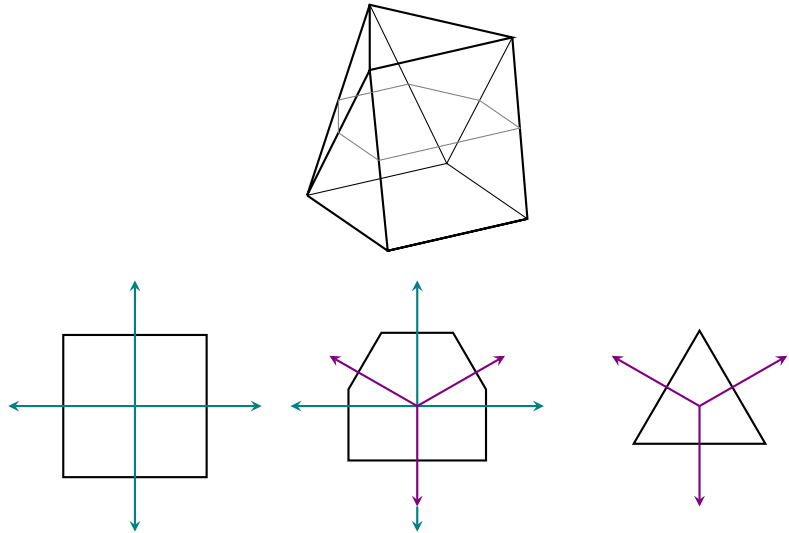

\begin{proof}
For the first part, Every combinatorial prismoid is a bounded poset, and every face can be uniquely identified by its intersection with the bases, and the bases are faces, so every combinatorial prismoid is isomorphic to an abstract prismoid.  
The second part follows from the fact that we can project a prismoid so that its bases are in parallel hyperplanes, in which case a horizontal slice between these hyperplanes is a weighted Minkowski sum of the bases, and the normal fan of the Minkowski sum of a pair of polytopes is the common refinement of their normal fans \cite[Proposition 7.12]{ziegler2007lecture}.

Specifically, for the `if' direction of the second part, suppose we have such realizations $B_0$, $B_1$.  Recall that every face of a polytope (in particular $P = [B_0 ; 0] \cv [B_1 ; 1] \subset \mb{R}^d $) is the solution set of some linear optimization problem.  Specifically, for a face $f$ these are the optimization problems with linear objective function in the relative interior of the normal cone of $f$.  Notice that the optimal solutions in $[B_i ; i]$ depend only on the restriction of the linear objective function $v^* = [w^* \ c]$ to the first $d-1$ coordinates.  
For $q_i \in [B_i ; i]$, we have $v^* q_0 = w^* q_0$ and $v^* q_1 = w^* q_1 + c$. 
For optimal solutions $q_i$ in $B_i$ to objective $w^*$, setting $c=w^* (q_0 - q_1)$ we get $v^* q_0 = v^* q_1$. 
Hence, a non-trivial pair of faces $(f_0,f_1)$ defines a face of $P$ if and only if there is a vector in the relative interior of the normal cone of the faces $f_0$ of $B_0$ and $f_1$ of $B_1$. 
Therefore, $\side(\labl(P))$ is the common refinement of the normal fans of $B_0$ and $B_1$. 

For the `only if' direction suppose we have a realization $P$ of some abstract prismoid $\mc{P} \subset \mc{B}_0 \catprod \mc{B}_1$.  Let $B_0 = \face(P,(\top,\bot))$ and $B_1 = \face(P,(\bot,\top))$.
Then there is some projective transformation $\pi$ sending $B_i$ into the hyperplane $\{x:x_d=i\}$. 
Hence $\pi(P) = [B_0' ; 0] \cv [B_1' ; 1]$ for some realizations $B_i'$ of $\mc{B}_i$, and again 
$\side(\mc{P})$ is the common refinement of the normal fans of $B_0'$ and $B_1'$. 
\end{proof}

Recall from the introduction that an (antiprism) interval polytope of $P$ is defined as a polytope with face lattice consisting of the intervals of $\labl(P)$ ordered by (reverse) inclusion.  
The \df{abstract antiprism} of a bounded poset $\mc{P}$ is the poset $\{ (g,f^*) \in \mc{P} \catprod \mc{P}^* : g \leq f \}$.
When an abstract antiprism can be realized we call it a combinatorial antiprism. 

We call an ordered pair of polytopes $(P_1,P_2)$ \df{balanced}, when they are centered, have the same combinatorial type, and the relative open normal cone of a face $g$ of $P_1$ intersects the relative open face $f$ of $P_2$ if and only if $f$ is greater than $g$, 
\[ \bal(P_1,P_2) \ := \ \forall g,f \neq \bot. \ (\ncone(P_1,g)^\circ \cap \face(P_2,f)^\circ \neq \emptyset  \ \Leftrightarrow \ g \leq f) \ . \]
Observe that $\bal(P_1,P_2)$ implies $\bal(P_2^*,P_1^*)$, but that balance is not a symmetric relation, see Figure~\ref{fig:balance}. 

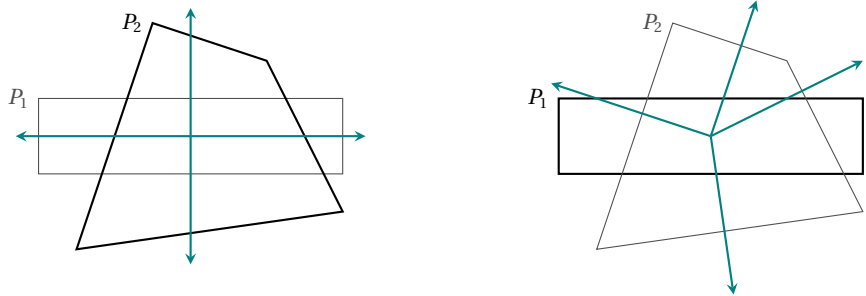
\begin{figure}[ht]
\centering

\begin{tikzpicture}[scale=0.5]

\path
 (4,1) coordinate (a)
 (4,-1) coordinate (b)
 (-4,-1) coordinate (c)
 (-4,1) coordinate (d)
 (4,-2) coordinate (e)
 (2,2) coordinate (f)
 (-1,3) coordinate (g)
 (-3,-3) coordinate (h)
;

\matrix[column sep=2cm]
{

\draw[black!70]
 (a) -- (b) -- (c) -- (d) node[left] {\figsz $P_1$} -- cycle
;
\draw[thick]
 (e) -- (f) -- (g) node[left] {\figsz $P_2$} -- (h) -- cycle
;

\begin{scope}[thick,teal,->]
\draw (0,0) -- (2.3,0);
\draw (0,0) -- (0,1.7);
\draw (0,0) -- (-2.3,0);
\draw (0,0) -- (0,-1.7);
\end{scope}

&

\draw[thick]
 (a) -- (b) -- (c) -- (d) node[left] {\figsz $P_1$} -- cycle
;
\draw[black!70]
 (e) -- (f) -- (g) node[left] {\figsz $P_2$} -- (h) -- cycle
;

\begin{scope}[thick,teal,->]
\draw[scale=0.6] (0,0) -> (1,3);
\draw[scale=0.5] (0,0) -> (4,2);
\draw[scale=0.3] (0,0) -> (1,-7);
\draw[scale=0.7] (0,0) -> (-3,1);
\end{scope}

\\
};

\end{tikzpicture}
\caption{
\textsc{Left:} A pair of quadrilaterals that are balanced, $\bal(P_1,P_2)$. 
\textsc{Right:} The same pair of quadrilaterals in reverse order is not balanced, $\neg\bal(P_2,P_1)$.  
}\label{fig:balance}
\end{figure}

\begin{thrm}\label{the:bpdef}
A combinatorial polytope has an antiprism if and only if it has a balanced pair. 
Moreover, if $\bal(P_0,P_1)$ then $[P_0 ; 0] \cv [P_1^* ; 1]$ is an antiprism of $P_0$. 
\end{thrm}

\begin{proof}
By Lemma \ref{lem:prism}, $\mc{P}$ has an antiprism if and only if there is 
are realizations $P_0$ of $\mc{P}$ and $P_1^*$ of $\mc{P}^*$ such that the common refinement of their normal fans are indexed by minimal pairs $(g,f^*) \in {\mc{P} \catprod \mc{P}^*}$ such that $g \leq f$.
And, in this case $[P_0 ; 0] \cv [P_1^* ; 1]$ is an antiprism. 
The common refinement of the normal fans consists of minimal non-empty intersections, so this is equivalent to the statement, 
``$\ncone(P_0,g)^\circ$ and $\ncone(P_1^*,f^*)^\circ$ intersect if and only if $g \leq f$''. 
Since $\ncone(P_1^*,f^*)^\circ = \cone(P_1,f)^\circ = \mb{R}_{\geq 0} \face(P_1,f)^\circ$ and $P_1$ realizes $\mc{P}$, this is equivalent to $\bal(P_0, P_1)$.
\end{proof}

We say a polytope is \df{perfectly centered} when 
the orthogonal projection of the origin into the affine closure of each face is in the relative interior of that face.  
In particular, if a perfectly centered polytope has full dimension, then it is centered. 
Note this has the following physical interpretation.  A polytope is perfectly centered when it can rest on any face without toppling over, assuming its center of mass is at the origin. 
This is also equivalent to the polytope being balanced with itself. 

\begin{lem}
\label{lem:pc=balself}
A polytope is perfectly centered if and only if it is balanced with itself. 
\end{lem}

\begin{proof} 
By \cite[Theorem 2.1]{broadie1985antiprisms}, a polytope $P$ is perfectly centered if and only if $[P ; 0] \cv [P^* ; 1]$ is an antiprism of $P$, and by Theorem \ref{the:bpdef}, this holds if and only if $\bal(P,P)$.
\end{proof}

Unlike 4 dimensions, where the problem of realizing polytopes has no easy solution, the situation is much simpler in 3 dimensions \cite{richter1996realization}.  Every 3-polytope has a particularly nice realization called a midscribed polytope \cite{thurston1978geometry}. 
Recall that a 3-polytope is midscribed when every edge is tangent to the unit sphere.  

\begin{thrm}
Midscribed polytopes are perfectly centered. 
Hence, every combinatorial 3-polytope has an antiprism.  
\end{thrm}

\begin{proof}
Each edge of a midscribed polytope is tangent to the unit sphere, so the orthogonal projection of the origin into the line spanning that edge is exactly the point of tangency.  
Hence, the perfectly centered condition holds for edges. 
For each facet of a midscribed polytope, the plane spanning that facet intersects the unit ball in a disk, and each edge of the facet is tangent to the disk. 
In other words, each facet circumscribes the disk where it intersect the unit ball. 
The orthogonal projection of the origin into the plane spanning the facet is the center of this disk. 
Hence, the condition holds for facets.
Since midscribed polytopes correspond to disk packings of the sphere in this way, we may assume the polytope is centered, otherwise apply an appropriate conformal map so that the disks are not all contained in any single hemisphere \cite{thurston1978geometry}.
For vertices, the condition is trivial. 
Thus, midscribed polytopes are perfectly centered, and since 
every combinatorial 3-polytope has such a realization, 
by Lemma \ref{lem:pc=balself} every 3-polytope has a realization that is balanced with itself,
so by Theorem \ref{the:bpdef} every combinatorial 3-polytope has an antiprism. 
\end{proof}

We now set about constructing a 4-dimensional polytope without an antiprism. 
By Lemma~\ref{lem:prism} this is equivalent to finding a combinatorial 4-polytope without a balanced pair. 
We use Theorem~\ref{thrm:log_comb=proj} to reduce this problem to finding an polygon that cannot be balanced by projective transformations, Lemma~\ref{lem:unbalp}.  To use Theorem~\ref{thrm:log_comb=proj}, we need to show that faces projectively inherit $\bal$, Lemma~\ref{lem:bpi}, and that $\bal$ is an semialgebraic property, which may be observed directly from the definition of $\bal$. 

Since $\bal$ is defined in terms of polar duality, it will be helpful to recall how the polar dual of a polytope is transforms when applying a projective transformation in the primal. 
We represent a projective transformation $\pi$ acting on a vector $x \in \mb{R}^d$ by a matrix $M = [A,  b ; c^*,  1]$ acting on homogeneous coordinates by 
\[ \pi(x) =  \left[ \begin{array}{cc} A & b \\ c^* & 1 \\ \end{array}\right]_\mb{P} (x) = \frac{Ax +b}{c^*x+1}. \]
For $\pi$ as above, we denote 
\[ \pi^* 
= \left(\left[ \begin{array}{cc} -I & 0 \\ 0 & 1 \\ \end{array}\right] M^* \left[ \begin{array}{cc} -I & 0 \\ 0 & 1 \\ \end{array}\right] \right)_\mb{P}
=  \left[ \begin{array}{cc} A^* & {-}c \\ {-}b^* & 1 \\ \end{array}\right]_\mb{P}, \]
and we call $\pi^{-*} := (\pi^*)^{-1} = (\pi^{-1})^*$ the polar transformation of $\pi$.

\begin{prop}\label{prop:polarproj}
For a centered polytope $P$ and a projectivity $\pi$ such that $\pi(P)$ is centered, bounded, and has the same orientation as $P$,
\[ \pi(P)^* = \pi^{-*}(P^*). \]
\end{prop}

\begin{proof}
Since $\pi$ is bounded and preserves orientation on $P$, we have $\forall p\in P.\ c^*p+1 > 0$. 
Since $\forall x\in \pi(P)^*, \forall p \in P^\circ.\ \left<x,\pi p\right> < 1$, 
and in particular $\left<x,\pi \vec{0}\right> = x^*b < 1$, 
we have $-b^*x+1 > 0$.
Thus, 
\[ \begin{array}{r@{\ }c@{\ }l}
\pi(P)^* 
&=& \left\{ x: \forall q\in \pi(P), \left<x, q\right> \leq 1 \right\} \\
&=& \left\{ x: \forall p\in P, \left<x,\pi p\right> \leq 1 \right\} \\
&=& \left\{ x: \forall p\in P, \left<x, \frac{A p + b}{c^*p+1}\right> \leq 1 \right\} \\
&=& \left\{ x: \forall p\in P, \left<x, A p \right> \leq \left<c,p\right> +1 - \left<x,b\right> \right\} \\
&=& \left\{ x: \forall p\in P, \left<\frac{A^* x -c}{-b^*x +1}, p \right> \leq 1 \right\} \\
&=& \left\{ x: \forall p\in P, \left<\pi^{*}x, p\right> \leq 1 \right\} \\
&=& \left\{ \pi^{-*}(y): \forall p\in P, \left<y, p\right> \leq 1 \right\} \\
&=& \pi^{-*}(P^*).  \qedhere
\end{array}  \]
\end{proof}

Suppose we are given a pair of polygons, and we want to find a balanced pair of projective copies. 
We start by reducing the space of transformations we need to consider. 
Since applying a projective transformation to one polygon and applying the polar transformation to the other polygon preserves balance, 
we can reduce the problem to applying a projective transformation to only one polygon and keeping the other polygon fixed.  
Furthermore, 
a projective transformation consists of an affine part and a perspectivity, but a perspectivity applied to a single vector only scales that vector.  
To see this, compare the matrix representation of a projective transformation to that of the affine part of the same transformation 
\[
\left[ \begin{array}{cc}
A & b \\ c^* & 1 \\ \end{array}\right]_\mb{P} (x) = \frac{Ax +b}{c^*x+1}, 
\hspace{2cm} 
\left[ \begin{array}{cc}
A & b \\ 0 & 1 \\ \end{array}\right]_\mb{P} (x) = \frac{Ax +b}{1}.
\] 
Since scaling vectors by positive values does not change the cone of positive linear combinations of those vectors, and balance depends on the intersection of cones, we can reduce the problem further to balancing a pair of polygons by applying an affine transformation to one of the polygons.

We will now see informally why an affine transformation cannot always balance polygons. 
Two polygons $P_0,P_1$ are balanced if and only if 
the direction vectors of the vertices of $P_0$ and $P_1^*$ are interleaved around the unit circle.  
That is, the vectors alternate around the unit circle between belonging to one polygon and the polar of the other. 
For an affine transformation to balance one polygon with another,  
it may have to change some of these direction vectors to make them alternate. 
If we think of moving a transformation continuously from the identity to one that balances the pair, along the way the vectors will turn clockwise or counter-clockwise on the circle. 
The main idea is to construct a pair of polygons that require such an affine transformation to turn too many direction vectors alternately clockwise and counter-clockwise. 


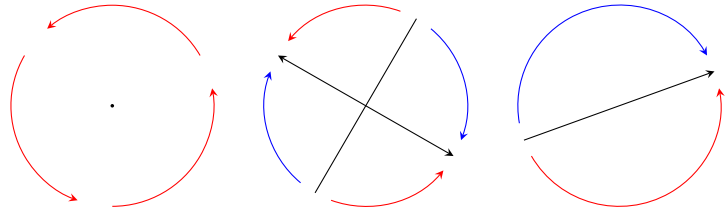
\begin{figure}[h]
\begin{center}

\begin{tikzpicture}[scale=2/3]
\draw[->,red] (30:2) arc (30:130:2);
\draw[->,red] (150:2) arc (150:250:2);
\draw[->,red] (-90:2) arc (-90:10:2);
\fill (0,0) circle (1pt); 

\path (5,0) coordinate (a);
\draw[->,red]  (a) +(70:2) arc (70:140:2);
\draw[<-,blue] (a) +(160:2) arc (160:230:2);
\draw[->,red]  (a) +(250:2) arc (250:320:2);
\draw[<-,blue] (a) +(-20:2) arc (-20:50:2);
\draw (a) +(60:2) -- +(240:2);
\draw[<->] (a) +(150:2) -- +(330:2);

\path (10,0) coordinate (a);
\draw[<-,blue] (a) +(30:2) arc (30:190:2); 
\draw[->,red] (a) +(-150:2) arc (-150:10:2);
\draw[<-] (a) +(20:2) -- +(200:2);

\end{tikzpicture}

\caption{From the left, the way the directions of vectors turn by rotating, stretching, translating.}
\label{fig:turning_vectors}
\end{center}
\end{figure}

To get an idea of how many direction vectors is too many we decompose an affine transformation into parts and see how many vectors each part can handle.  Assume the transformation preserves orientation. 
Consider the special orthogonal linear part (rotating), 
symmetric positive definite linear part (stretching), 
and translational part of an orientation preserving affine transformation.  
The orthogonal part turns all vectors in the same way. 
The spd part divides the circle into 4 quadrants where direction vectors turn alternately clockwise and counter-clockwise. 
And, the translational part divides the circle into 2 halves where direction vectors turn the opposite way.  
Naively adding this up we get 7 regions; see Figure \ref{fig:turning_vectors}.  We construct a polygon such that an affine transformation must to turn 8 vectors alternately clockwise and counter-clockwise to balance the polygon with a copy of itself, which we then show is impossible.

With these limitations in mind, 
let $G$ be the polygon with vertices $(8,5)$, $(7,7)$ and all permutations and changes of sign: $(8,5)$, $(7,7)$, $(5,8)$, $(-5,8)$, $(-7,7)$, $(-8,5)$, $(-8,-5)$, $(-7,-7)$, $(-5,-8)$, $(5,-8)$, $(7,-7)$, $(8,-5)$.

\begin{figure}[h]
\begin{center}
\begin{tikzpicture}[scale=1/2]

\draw[thick,->,teal]
(0,0) -- (9.6,4.8);
\draw[ultra thick,teal]
(7,7) -- (8,5)
;


\draw  (7,7) -- (5,8)
 -- (-5,8) -- (-7,7) -- (-8,5)
 -- (-8,-5) -- (-7,-7) -- (-5,-8)
 -- (5,-8) -- (7,-7) -- (8,-5)
 -- (8,5) -- cycle
;


 \draw[very thick,->,red] (25:6) arc (25:35:6);
 \draw[<-,blue] (55:6) arc (55:65:6);
 \draw[->,red] (115:6) arc (115:125:6);
 \draw[<-,blue] (145:6) arc (145:155:6);
 \draw[->,red] (205:6) arc (205:215:6);
 \draw[<-,blue] (235:6) arc (235:245:6);
 \draw[->,red] (295:6) arc (295:305:6);
 \draw[<-,blue] (325:6) arc (325:335:6);

\end{tikzpicture}
\caption{A polygon that cannot be balanced with itself by affine transformations.
The thick red arrow indicates how the normal ray of the thick edge must turn to intersect that edge.
Similarly, the 7 other red and blue arrows indicate how the normal rays of other edges must turn to balance an affine copy of the polygon with itself.}  
\end{center}
\end{figure}
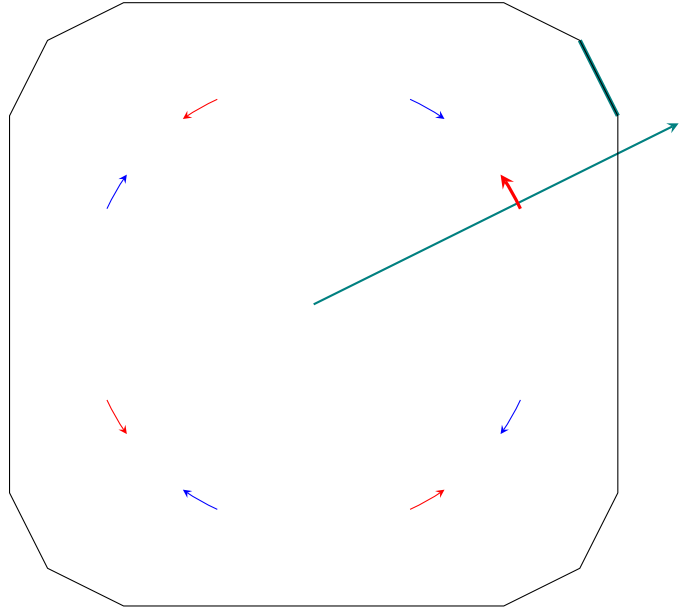

\begin{lem}\label{lem:unbalp} No two projective copies of $G$ are balanced. 
\end{lem}

\begin{proof} The polar $G^*$ has vertices $(\frac{1}{8},0)$, $(\frac{2}{21},\frac{1}{21})$ and all permutations and changes of sign. Balance is invariant under positive scaling, so we scale $G^*$ by 21 for convenience, giving vertices $(\frac{21}{8},0)$, $(2,1)$ instead.  We see that $G$ is not perfectly centered since that would require the slope $m$ of the outward normal vector of the edge between $(8,5)$ and $(7,7)$ to be between the slopes of the vertices, ${\frac{5}{8}< m < 1}$, but the slope is ${m=\frac{1}{2}<\frac{5}{8}}$ as seen from the vertex $(2,1)$ of $21 G^*$.  By construction, the reflection group of $G$ and $G^*$ is the same as that of the unit square, and is given by the matrices 
\[ \begin{array}{cccc}
\left[ \begin{array}{cc}
1 & 0 \\
0 & 1 \\
\end{array} \right] &
\left[ \begin{array}{cc}
0 & 1 \\
1 & 0 \\
\end{array} \right] &
\left[ \begin{array}{cc}
0 & -1 \\
1 & 0 \\
\end{array} \right] &
\left[ \begin{array}{cc}
-1 & 0 \\
0 & 1 \\
\end{array} \right] \phantom{.} \\
\\
\left[ \begin{array}{cc}
-1 & 0 \\
0 & -1 \\
\end{array} \right] &
\left[ \begin{array}{cc}
0 & -1 \\
-1 & 0 \\
\end{array} \right] &
\left[ \begin{array}{cc}
0 & 1 \\
-1 & 0 \\
\end{array} \right] &
\left[ \begin{array}{cc}
1 & 0 \\
0 & -1 \\
\end{array} \right] . \\
\end{array} \]
These transformations give us a total of 8 places where the perfectly centered condition is violated.  
Consider an affine transformation $T$ acting on $21G^*$ by  
\[ T \left[\begin{array}{c} x \\ y \\ \end{array}\right]:=
\left[\begin{array}{cc}
a & b \\
c & d \\
\end{array}\right]
\left[\begin{array}{c} x \\ y \\ \end{array}\right]
+\left[\begin{array}{cc}
s \\
t \\
\end{array}\right] . \]
For $\bal(G,T^{-1} G) = \bal(T(21G^*)^*,G)$ to hold, 
we must have $a,d>0$.  
For $a > 0$, this is because $T \left( \frac{21}{8},0\right)$ must point to the right and $T \left( -\frac{21}{8},0\right)$ must point to the left, 
\[T\left( \tfrac{21}{8},0\right)_1= \tfrac{21}{8}a+s >0, \quad
  T\left(-\tfrac{21}{8},0\right)_1=-\tfrac{21}{8}a+s <0 . \]  
If ${s\leq 0}$ then the first inequality implies ${a>0}$, and if ${s\geq 0}$ then the second inequality implies ${a>0}$.  The same holds for $d$ because of the corresponding inequalities in the ${2^\text{nd}}$ coordinate. 

Additionally, the image of $(2,1)$ must have slope greater than ${\frac{5}{8}}$, and this must also be the case for $T$ conjugated by all elements of the reflection group.  
The image of $(2,1)$ by all conjugates is 
\[\begin{array}{cccc}
\left[\begin{array}{c}
 a2+b1+s \\
 c2+d1+t \\
\end{array}\right] &
\left[\begin{array}{c}
 d2+c1+t \\
 b2+a1+s \\
\end{array}\right] &
\left[\begin{array}{c}
 d2-c1-t \\
 -b2+a1+s \\
\end{array}\right] &
\left[\begin{array}{c}
 a2-b1-s \\
 -c2+d1+t \\
\end{array}\right] \phantom{.} \\
\\
\left[\begin{array}{c}
 a2+b1-s \\
 c2+d1-t \\
\end{array}\right] &
\left[\begin{array}{c}
 d2+c1-t \\
 b2+a1-s \\
\end{array}\right] &
\left[\begin{array}{c}
 d2-c1+t \\
 -b2+a1-s \\
\end{array}\right] &
\left[\begin{array}{c}
 a2-b1+s \\
 -c2+d1-t \\
\end{array}\right] . \\
\end{array} \]
For the first of these vectors the slope requirement is given by the following inequality
\[ \frac{T(2,1)_2}{T(2,1)_1} = \frac{c2+d1+t}{a2+b1+s}>\frac{5}{8}. \] 
Equivalently, $-10a-5b+16c+8d-5s+8t>0$. 
Putting the inequalities we get from all these slope requirements with the sign requirements of $a,d$ together we get the matrix inequality 
\[
\left[\begin{array}{cccccc}
-10 & -5 & 16 & 8 & -5 & 8 \\
8 & 16 & -5 & -10 & 8 & -5 \\
8 & -16 & 5 & -10 & 8 & 5 \\
-10 & 5 & -16 & 8 & 5 & 8 \\
-10 & -5 & 16 & 8 & 5 & -8 \\
8 & 16 & -5 & -10 & -8 & 5 \\
8 & -16 & 5 & -10 & -8 & -5 \\
-10 & 5 & -16 & 8 & -5 & -8 \\
1 & 0 & 0 & 0 & 0 & 0 \\
0 & 0 & 0 & 1 & 0 & 0 \\
\end{array}\right] 
\left[\begin{array}{c}
a \\ b \\ c \\ d \\ s \\ t \\
\end{array}\right]
> 
0 \ . \]
Finding a solution to this inequality amounts to finding a vector in the column space of the matrix that has all positive entries.
The columns of this matrix, however, are all orthogonal to ${[1;1;1;1;1;1;1;1;8;8]}$, which has all positive entries, so the column span of the matrix is outside of the open positive orthant, which implies that no values for $a,b,c,d,s,t$ satisfy all of these inequalities. Therefor, there is no affine transformation $T$ such that 
$\bal(G, T^{-1} G)$.     

If there were projective transformations $\pi_1$, $\pi_2$ such that $\bal(\pi_1(G), \pi_2(G))$, then we would have $\bal(G, \pi^{*}_1  \pi_2 (G))$, and the affine part of $\pi^{*}_1  \pi_2$ would balance $G$ with itself, which we have just seen to be impossible.  Thus, no two projective copies of $G$ are balanced. 
\end{proof}

We will now show that faces projectively inherit balance. 
For some intuition why, notice that an 
interval polytope of $\mc{P}$ has among its faces all of the interval polytopes of the faces of $\mc{P}$. 
By Theorem~\ref{the:bpdef}, if a combinatorial polytope $\mc{P}$ has a balanced pair, then it has an interval polytope, which means all of its faces also have interval polytopes. 
By the other direction of Theorem~\ref{the:bpdef}, we get a balanced pair for any face of $\mc{P}$.  
Hence, if a combinatorial polytope has a balanced pair, then each of its faces also has a balanced pair. 
But to show projective inheritance, we need slightly more.  We need the faces of a balanced pair $(P_1,P_2)$ to have projective copies that are balanced. 

To find projective copies of some face of a balanced pair $(P_1,P_2)$, we perform the natural geometric analog of the above argument.  We construct the antiprism $A$ with bases $P_1,P_2$, then the polar dual of a facet of $A^*$ will be an antiprism $A_f$ having balanced projective copies of the corresponding faces $f$ of $P_1,P_2$ as bases.  

Before proving Lemma~\ref{lem:bpi}, we present the basic algebraic rules relating face cones, normal cones, and polar duality for polytopes when working in subspace and direct sums of vector spaces.
For a cone $C$ in a vector space $V$, let $C^\diamond \subset \linspan(C)^*$ denote the polar dual of $C$ regarded as a polytope in the vector subspace spanned by $C$.  Clearly if $C$ has full dimension then $C^\diamond = C^*$. 
Given two vector spaces $V_1,V_2$, let $V_1 \oplus V_2$ denote the direct product.  
For $a_i \in V_i^*$ let $a_1 \oplus a_2$ be the linear functional on $V_1 \oplus V_2$ defined by 
$[a_1 \oplus a_2] (x_1, x_2) = a_1(x_1) + a_2(x_2) $.
Note that if $V_1,V_2 \subset V$ are complimentary vector subspaces, then $V_1 \oplus V_2 = V$, and we may consider $V_1^*$ as being orthogonal to $V_2$ in this sense that $V_1^* \oplus \vec{0} \subset V^*$ is the subspace of linear functionals that vanish on $V_2$.
Recall that the set of translations of an affine space $X$ form a vector space, which we denote $\trans(X)$.
For a polytope $P \subset X$, the tangent cone $\ncone(P,f)^* \subset \trans(X)$ is the cone generated by all translations that send some point on the face $\face(P,f)$ to a point in $P$, and the normal cone $\ncone(P,f) \subset \trans(X)^*$ is the polar dual of the tangent cone.
For a flat $V \supset P$ of $X$ and a point $p \in V$, let $i_{V,p}(P)$ denote $P$ regarded as a polytope in the vector space $V$ with origin $p$.

\begin{prop}\label{prop:rules}
For cones $C_1,C_2$ spanning complementary subspaces, 
\[ \begin{array}{r@{\ }c@{\ }l}
 \labl(C_1 + C_2) &=& \labl(C_1) \catprod \labl(C_2)  \\
 (C_1 + C_2)^* &=& C_1^\diamond \oplus C_2^\diamond \\
 \face(C_1 + C_2,(f_1,f_2)) &=& \face(C_1,f_1) + \face(C_2,f_2) \\
\end{array} \]
For a centered polytope $P$ with face indices $g \leq f$ and $r>0$,
\[ \begin{array}{r@{\ }c@{\ }l}
 \labl(\cone(P,f)) &=& [\bot, f] \\
 \face(\cone(P,f),g) &=& \cone(\face(P,g)) \\
 \ncone(P,f) &=& \cone(P^*,f^*) \\
 \ncone(\face(P,f),g) &=& \face(\ncone(P,g)^*,f)^* \\
 \cone([P; r])^* &=& \cone([P^*;  {-1}/{r}]). \\
\end{array} \]
For a polytope $P$ in affine space with a face $f$, a point $p \in P$, and a subspace $V \supset P$, 
\[ \begin{array}{r@{\ }c@{\ }l}
 \ncone(i_{V,p}(P),f)^* &=& p + \ncone(P,f)^*. \\
\end{array} \]
\end{prop}

We do not include a proof of Proposition~\ref{prop:rules} here.  
These rules may easily but tediously verified by the reader.

\begin{lem}\label{lem:bpi}
Faces projectively inherit the predicate $\bal$. 
\end{lem}

\begin{proof}
Let 
$P_1,P_2$ be a pair of centered realizations of $\mc{P}$ such that $\bal(P_1, P_2)$.  
By Lemma~\ref{lem:prism} $P_1,P_2$ are base faces of an antiprism $A = [P_1; 1] \cv [P_2^*; -1]$.
Note that the faces of $A$ are 
\[ \face(A,(g,f^*)) = \face\left([P_1; 1],g\right) \cv \face\left([P_2^*; -1],f^*\right) \]
for $g \leq f$,
and $\cone([P_1; 1],f)$ and $\cone([P_2^*; -1],f^*)$ are in complementary linear subspaces.

For any $f\in \mc{P}$, let 
$A_f^* = \face(A^*,(\top,f))$, and let $A_f = i_{V,p}(A_f^*)^*$
where $V$ is the affine span of $A_f^*$, and $p$ is some point in the relative interior of $A_f^*$.
The faces of $A_f$ are indexed by $(h,g^*) \in {[\bot,f] \catprod [f^*,\bot^*]} \subset {\mc{P} \catprod \mc{P}^*} $ 
such that $h \leq g$. 
Let $F_i = \face(P_i,f)$,
$\tilde F_1 = \face(A_f,(f,f^*))$,
$\tilde F_2^* = \face(A_f,(\bot,\top))$,
and 
$\tilde F_2 = i_{W,q}(\tilde F_2^*)^*$ where $W$ is the affine span of $\tilde F_2^*$ and $q$ is some point in the relative interior of $\tilde F_2^*$. 
Observe that $A_f$ is an antiprism of the combinatorial polytope $[\bot,f]$, so 
by Lemma~\ref{lem:prism} again, we have $\bal(\tilde F_1, \tilde F_2)$. 

We claim that 
$\tilde{F}_1, \tilde{F}_2$ are projective copies of $F_1,F_2$ respectively.

\[ \begin{array}{r@{\ }c@{\ }l@{\vspace{3pt}}}
\cone(\tilde F_1)^*
& = & \cone(A_f,(f,f^*))^* \\
& = & \cone(i_{V,p}(\face(A^*,(\top,f)))^*,(f,f^*))^* \\
& = & \ncone(i_{V,p}(\face(A^*,(\top,f))),(f^*,f))^* \\
& = & p+\ncone(\face(A^*,(\top,f)),(f^*,f))^* \\
& = & p+\face(\ncone(A^*,(f^*,f))^*,(\top,f)) \\
& = & p+\face(\cone(A,(f,f^*))^*,(\top,f)) \\
& = & p+\face((\cone(A,(f,\bot))+\cone(A,(\bot,f^*)))^*,(\top,f)) \\
& = & p+\face(\cone(A,(f,\bot))^\diamond\oplus\cone(A,(\bot,f^*))^\diamond,(\top,f)) \\
& = & p+\face(\cone(A,(f,\bot))^\diamond,\top)\oplus\face(\cone(A,(\bot,f^*))^\diamond,f) \\
& = & p+\cone([F_1 ; 1])^\diamond\oplus\vec{0}. \\
\end{array} \]
Since $V$ is a translate of the orthogonal compliment of $\cone(A,(\bot,f^*))$, we have $\cone(\tilde F_1) = {\proj_V(p+\cone([F_1 ; 1]))}$ where $\proj_V$ denotes orthogonal projection into $V$. Thus, $\tilde{F}_1$ and $F_1$ are projectively equivalent. 

\[ \begin{array}{r@{\ }c@{\ }l@{\vspace{3pt}}}
\cone(\tilde F_2^*)^*
& = & \cone(A_f,(\bot,\top))^* \\
& = & p+\face((\cone(A,(\bot,\bot))+\cone(A,(\bot,\top)))^*,(\top,f)) \\
& = & p+\face((\vec{0} +\cone([P_2^* ; -1]))^*,(\top,f)) \\
& = & p+\face(\cone([P_2^* ; -1])^*,f) \\
& = & p+\face(\cone([P_2 ; 1]),f) \\
& = & p+\cone([F_2 ; 1]) \\
\end{array}. \]
We also have 
\[ \cone(\tilde F_2^*)^* = \cone\left(i_{V,p}(\tilde F_2 ) -(1+\|q\|^{-2})q\right) ,\] 
where $q$ is regarded as a vector in $V$ with origin $p$.
Thus, $\tilde{F}_2$ and $F_2$ are projectively equivalent, so the claim holds.  
This implies that faces projectively inherit $\bal$. 
\end{proof}

\begin{thrm}
\label{the:bai}
In dimensions 4 or more there exists a combinatorial polytope that does not have a balanced pair, antiprism, or interval polytope. 
\end{thrm}

\begin{proof}
By Lemma~\ref{lem:bpi}, we can use Theorem~\ref{thrm:log_comb=proj} to reduce this problem to finding a polygon that cannot be balanced by projective transformations, and by Lemma~\ref{lem:unbalp}, $G$ is exactly such a polygon.  Hence, there is a combinatorial 4-polytope that does not have a balanced pair. 
Specifically, the stamp $\mc{S}_G$ from Theorem~\ref{the:paf} is such a polytope. 
And, by Theorem~\ref{the:bpdef} this polytope does not have an antiprism or an interval polytope. 
\end{proof}

\section{Stamps}
\label{stamp}

In this section we construct the stamp of a polytope.  This construction uses modified versions of constructions from Richter-Gebert's proof of universality for 4-polytopes \cite{richter1996realization}. 
The main tool we take is a method for encoding polynomials into a combinatorial polytope in such way that realizations of this combinatorial polytope correspond to solutions of the given polynomial. 
It may be helpful for the reader to have a copy of \cite{richter1996realization} on hand. 

One approach to constructing the stamp $\mc{S}_P$ of a polytope $P$ with algebraic coordinates is to first 
encode the minimal polynomial of each coordinate into a combinatorial polytope, using the same methods as the proof of universality.  
We additionally encode bounds on each coordinate that are sufficiently small to determine a unique root of the minimal polynomial.  
Then, we combine these combinatorial polytopes into a single combinatorial polytope $\mc{S}_P$ in a way that forces each coordinate of the  vertices of a certain face of $\mc{S}_P$ to equal the corresponding encoded roots. 
The actual construction will differ from this for reasons that will become clear in the proof, but the basic idea is similar to this approach.

\subsection{Gluing and Whittling} 

Throughout this section we will build up large combinatorial polytopes by combining smaller combinatorial polytopes.  The way we combine these is by gluing, and to a lesser extent whittling. 
Given a pair of combinatorial polytopes (or more generally bounded posets) $\mc{P}_0$ $\mc{P}_1$, a facet (coatom) $f_i \in \mc{P}_i$ of each, and an isomorphism $\varphi : [\bot, f_0] \to [\bot, f_1]$ between the respective face lattices of these facets, 
we \df{glue} $\mc{P}_0$ and $\mc{P}_1$ along these facets by removing the facets from each and identifying corresponding faces by the given isomorphism,
\[ \mc{P}_0\#_\varphi\mc{P}_1 \ 
:= \ (\mc{P}_0\setminus f_0)\sqcup(\mc{P}_1\setminus f_1)/\sim_\varphi \]
where $g_0 \sim_\varphi g_1$ when $\varphi(g_0) = g_1$ and $\top \sim_\varphi \top$, and  
the partial order on these equivalence classes is given by $\mc{G} \leq \mc{F}$ when 
$\exists g \in \mc{G}, \exists f \in \mc{F}.\: g \leq f$. 
We indicate this by a gluing diagram, which consists of a box for each poset and an edge between posets that are glued together, 
\[
\begin{tikzpicture}[baseline=(x0.base)]
\node at (0,0) [draw] (x0) {$\mc{P}_0$};
\node at (1,0) [draw] (x1) {$\mc{P}_1$};
\draw (x0) -- (x1);
\end{tikzpicture} \ \
:= \ \mc{P}_0\#_\varphi\mc{P}_1 . \] 
We see examples of polytopes glued together in 3 dimensions in Figure \ref{fig:connector}. 
We call the dual operation \df{whittling}, denoted $ \mc{P}_0\#^*_{\varphi}\mc{P}_1 := (\mc{P}_0^*\#_{\varphi}\mc{P}_1^*)^*$, and we say the vertex $v$ of $\mc{P}_0$ was whittled when $[v,\top]$ is the domain of $\varphi$. 
We will generally glue combinatorial polytopes along facets that are necessarily flat. 
A $d$-polytope is \df{necessarily flat} when any realization of its $(d\!-\!1)$-skeleton in a space of arbitrarily high dimension will be contained in a subspace of dimension $d$.  The following results are shown in \cite[Section 3.2]{richter1996realization}.

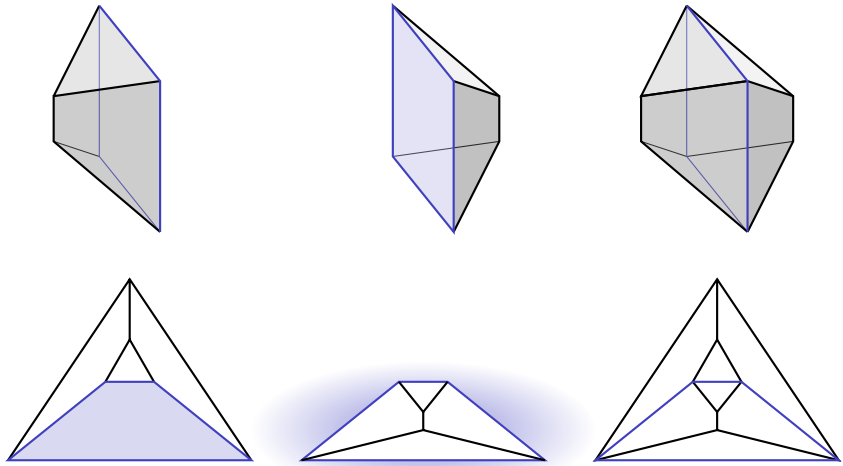
\begin{figure}[htb]

\centering
\begin{tikzpicture}

\colorlet{mycolor}{blue!50!gray}

\begin{scope}[scale=1]
 \path (-.4,0) coordinate (a)
 +(.8,-1) coordinate (b)
 +(.8,1)  coordinate (c)
 +(0,2) coordinate (d) 
 +(-.6,.2) coordinate (e) 
 +(-.6,.8) coordinate (f) 
 +(1.4,.2) coordinate (g)
 +(1.4,.8) coordinate (h);
\end{scope}

\matrix[row sep=6mm]
{
 \draw
  (a) -- (e); 
 \draw[mycolor] (d) -- (a) -- (b);
 \fill[fill=black!24,opacity=.4] (d) -- (f) -- (c);
 \fill[fill=black!48,opacity=.4] (b) -- (e) -- (f) -- (c); 
 \draw[thick] (b) -- (e) -- (f) -- (c)
  (d) -- (f);
 \draw[thick,mycolor] (b) -- (c) -- (d);
&

 \draw  (a) -- (g);
 \fill[black!12,opacity=.4]  (d) -- (h) -- (c);
 \draw[thick]  (d) -- (h);
 \filldraw[thick,fill=black!60,fill opacity=.4]  (b) -- (g) -- (h) -- (c);
 \filldraw[thick,mycolor,fill=mycolor!36,fill opacity=.4]  (b) -- (a) -- (d) -- (c) --cycle;
&

 \draw  +(e) -- +(a) -- +(g);
 \draw  +(c) -- +(f); 
 \draw[mycolor] (d) -- (a) -- (b);
 \fill[black!12,opacity=.4]  (d) -- (h) -- (c);
 \filldraw[thick,fill=black!24,fill opacity=.4]  +(d) -- +(f) -- +(c);
 \filldraw[thick,fill=black!48,fill opacity=.4] +(b) -- +(e) -- +(f) -- +(c); 
 \filldraw[thick,fill=black!60,fill opacity=.4] +(b) -- +(g) -- +(h) -- +(c);
 \draw[thick]  (d) -- (h);
 \draw[thick,mycolor] (b) -- (c) -- (d);
\\

\begin{scope}[scale=4/5]
 \path (-2,-.5) coordinate (a)
 +(4,0) coordinate (b)
 +(2.4,1.3) coordinate (c)
 +(1.6,1.3) coordinate (d)
 +(2,3) coordinate (e)
 +(2,2) coordinate (f)
 +(2,.5) coordinate (g)
 +(2,.8) coordinate (h);
\end{scope}

 \draw[thick] (a) -- (e) -- (b)
 (c) -- (f) -- (d)
 (e) -- (f); 
 \filldraw[thick,mycolor,fill=mycolor!20] +(a) -- +(b) -- +(c) -- +(d) -- cycle; 
&

 \shade[inner color=mycolor!60,outer color=white] (0,0) ellipse (2.25 and .9); 
 \filldraw[thick,draw=mycolor,fill=white]  +(a) -- +(b) -- +(c) -- +(d) -- cycle; 
 \draw[thick]  +(a) -- +(g) -- +(b)
 +(c) -- +(h) -- +(d)
 +(g) -- +(h); 
&

 \draw[thick]  +(a) -- +(e) -- +(b) -- +(g) -- cycle
 +(c) -- +(f) -- +(d) -- +(h) -- cycle
 +(e) -- +(f)
 +(g) -- +(h);
 \draw[thick,mycolor] +(a) -- +(b) -- +(c) -- +(d) -- cycle; 
\\
};

\end{tikzpicture}
\caption{
\textsc{Top:} A pair of polytopes and the result of gluing the pair together 
\textsc{Bottom:} A Schlegel diagram of each polytope.
}\label{fig:glue}
\end{figure}

\begin{lem}\label{lem:glue}
Any pair of polytopes with projectively equivalent facets, can be glued along those facets.\n 
Formally, for any facets $F_i$ of $P_i$ such that $F_1 \projeq F_2$, there exists a projectivity $\pi$ such that $\pi F_1 = F_2$ and 
\[\labl(\pi P_1 \cup P_2) = \labl(P_1) \#_{\varphi} \labl(P_1) \]
where $\varphi: \labl(F_1) \to \labl(F_2)$ is the isomorphism induced by $\pi$.
\end{lem}

\begin{lem}\label{lem:unglue}
Any realization of a pair of combinatorial polytopes glued along a necessarily flat facet of each can be decomposed into the union of realizations of the given pair that intersect along the given facets.\n
Formally, for any facets $f_i$ of $\mc{P}_i$ and isomorphism $\varphi : [\bot, f_1] \to [\bot, f_2]$, if $[\bot, f_1]$ is necessarily flat then for all realizations $P$ of $\mc{P}_0\#_\varphi\mc{P}_1$, $P = P_1 \cup P_2$ where $P_i$ realizes $\mc{P}_i$ and $P_1 \cap P_2 = \face(P_1,f_1) = \face(P_2,f_2)$. 
\end{lem}

\begin{lem}
Pyramids and Prisms of dimension at least 3 are necessarily flat.
\end{lem}

\subsection{Completion Conditions} 

The defining property of a stamp polytope is that a certain face is rigid up to projective transformations.  We will construct stamps from other polytopes with faces that are less constrained by gluing these polytopes together in ways that combine these constraints.  We call such constraints, completion conditions. 
Informally, the completion condition from a face to a polytope is the condition a realization of the face must satisfy to be completed to a realization of the entire polytope. 
Formally, for a bounded poset $\mc{P}$ and $f\in \mc{P}$, we say a realization $F$ of $[\bot,f]$ can be completed from $f$ to $\mc{P}$ when there exists a realization $P$ of $\mc{P}$ such that $F = \face(P,f)$, and 
we call any sequence of statements $\Gamma$ about a polytope $F$ the \df{completion condition} for $F$ from $f$ to $\mc{P}$ when $\Gamma$ holds if and only if 
$F$ can be completed from $f$ to $\mc{P}$. 
For example, the completion conditions for $F$ from $f_P$ to the stamp $\mc{S}_P$ is $F \projeq P$. 
For another example, if $\mc{P}$ is a bounded poset and $[\bot,f]$ is a combinatorial polytope, but $\mc{P}$ is not a combinatorial polytope, then the completion condition from $f$ to $\mc{P}$ is the logical sentence `False'. 

Similarly we define completion conditions for a collection of faces.  
Since we will only be concerned with the geometric properties of the faces, and not how they are positioned relative to eachother, and since projective equivalence is finer than combinatorial equivalence, we only consider each face up to its projective type. 
For $f_1,\dots,f_n \in \mc{P}$, we say respective realizations $F_1, \dots, F_n$ of $[\bot,f_1], \dots, [\bot, f_n]$ can be completed to $\mc{P}$ when there exists a realization $P$ of $\mc{P}$ and projectivities $\pi_i$ such that $\pi_i(F_i) = \face(P,f_i)$, and we call any sequence of statements $\Gamma$ about a collection of polytopes $F_1, \dots, F_n$ the \df{completion condition} from $f_1,\dots,f_n$ to $\mc{P}$ when $\Gamma$ holds if and only if these polytopes can be completed to $\mc{P}$.

\subsection{Visibility and Projective Space} 

Some of the completion conditions that will be used in the remainder of the paper 
may be understood geometrically as saying that certain faces of a polytope are visible from certain points in space around that polytope. 
We say a face is front visible (or simply ``visible'') from a point when that point can move along a straight line to any point on the face without hitting the polytope before reaching its destination. 
Formally, we define the visibility in terms of the convex join.  We say a face $f$ of $P$ is \df{front visible} from a point $p\nin P$ when the relative interior of the convex join of the face with the point is disjoint from the polytope, 
$ (p \cv \face(P,f))^\circ \cap P = \emptyset $.  
When a face is not front visible we say it is \df{front obscured}. 

We say a face $f$ of $P$ is back visible from $p$ when the point $p$ can move along a straight line in $\mb{R}^d$ to any point on the face $f$ of $P$ by first moving away from $P$, passing through the horizon at infinity, and then continuing towards $P$ from the opposite direction without hitting $P$ before reaching its destination.  
Formally, we define a face to be back visible from a point $p$ when it is front visible from the antipodal point in homogeneous coordinates. 
Observe that in homogeneous coordinates, a face $f$ of $P$ is front visible from a point $p\nin P$ when $ (\{r_1 [p;1] + r_2 \face([P;1],f) : r_i \geq 0\})^\circ \cap [P;1] = \emptyset $.  We say the face $f$ is \df{back visible} from $p$ when $ (\{{-}r_1 [p;1] + r_2 \face([P;1],f) : r_i \geq 0\})^\circ \cap [P;1] = \emptyset $.  Otherwise, we say the face is \df{back obscured}.
All together, this gives four possible answers to whether a face $f$ of $P$ is visible from a point $p$, which we denote as follows: 
\[ \vis(p,P,f) = 
\left\{\begin{array}{cl}
* & \text{doubly visible} \\
+ & \text{front only visible} \\
- & \text{back only visible} \\
0 & \text{doubly obscured.} \\
\end{array}\right.
\]
For $p \in \mb{R}^d$, let $p^-$ denote the point $p$ but with front and back visibility exchanged 
\[\vis(p^-,P,f) = -\vis(p,P,f),\] 
and let $(p^-)^- = p^+ = p$. 
Observe that if a projectivity is bounded and preserves orientation on $p \cv P$, then it preserves visibility for the pair $(p,P)$. 
If a projectivity is bounded and preserves orientation on $P$, but reverses orientation on $p$, then it exchanges front and back visibility.


We make use of two compactifications of $\mb{R}^d$.  First we compactify $\mb{R}^d$ by adjoining the sphere at infinity,
\[ \overline{\mb R}{}^d := \mb{R}^d \cup \left\{ \infty u  : u \in \mb{S}^{d-1} \right\} \]
where $\lim_{t \in \mb{N}} v_t = \infty u$ when 
$\lim_{t \in \mb{N}} [v_t;1]/\|[v_t;1]\| =  [u;0]$. 
We say the points $\infty u$ are on the \df{horizon}. 
For $v \in \mb{R}^d \setminus \vec{0}$, let $\infty v = \lim_{r \to \infty} rv = \infty(v/\|v\|)$. 
For $p = \infty u$ on the horizon, we exchange visibility by sending $p$ to the antipodal point on the horizon $p^- = -\infty u$. 
We denote the closure of $\overline{\mb R}{}^d$ under exchange of visibility by 
\[ \mb{O}^d := \overline{\mb R}{}^d \cup \{x^- : x \in \mb{R}^d \}  \]
where $\lim_{t \in \mb{N}} v_t^- = (\lim_{t \in \mb{N}} v_t)^-$.
Observe that $\mb{O}^d$ is homeomorphic to the sphere $\mb{S}^d$ by the map
\[ 
i_\mb{S} : \mb{O}^d \to \mb{S}^d, \quad
i_\mb{S}(x) := \left\{\begin{array}{rl@{\vspace{3pt}}} 
\tfrac{1}{\|[x;1]\|} [x;1] & x \in \mb{R}^d \\
\tfrac{{-}1}{\|[y;1]\|} [y;1] & x = y^-,\ y \in \mb{R}^d \\
 {[ u ; 0 ]} & x = \infty u,\ u \in \mb{S}^{d-1}. 
\end{array}\right.
\]

We additionally compactify $\mb{R}^d$ to the real projective space $\mb{P}^d := \mb{O}^d/\mb{Z}_2$,
the quotient of $\mb{O}^d$ by $p \equiv p^-$, which corresponds to the antipodal relation on the sphere, $i_\mb{S} (p^-) = {-}i_\mb{S} (p)$. 
In this context we call $\mb{O}^d$ \df{oriented projective space} \cite{stolfi1991oriented}. 
In projective space we have the advantage that the set of all flats form a complete lattice PG${}^d \subset 2^{\mb{P}^d}$ (the projective geometry on $\mb{R}^{d+1}$).  We denote the lattice closure and lattice operations by
\[ \projcl : 2^{\mb{O}^d} \to \text{PG}^d, \quad \projcl (X) = i_\mb{S}^{-1}(\linspan i_\mb{S} (X)\cap\mb{S}^d), \]
\[ {\vee},{\wedge} : \text{PG}^d \times \text{PG}^d \to \text{PG}^d, \quad 
a \vee b = \projcl (a \cup b), \quad a \wedge b = a \cap b. \]
real projective space has one more fundamental operation, cross ratio. 
For points $p$, $p_0$, $p_1$, $p_\infty$ on a line $\ell$ in $\mb{P}^d$, 
there is a the unique projectivity $\phi : \ell \to \mb{P}^1$ such that $\phi p_0 = 0$, $\phi p_1 = 1$, $\phi p_\infty = \infty$. 
The \df{cross ratio} 
$(p,p_1 | p_0, p_\infty)$ is the value of $\phi p \in \mb{P}^1 = \mb{R} \cup \infty$.

A basis of real projective space $\mb{P}^d$ is provided by a generic choice of $d{+}1$ point and 1 hyperplane.
There is a unique projectivity sending these these basis elements respectively to the origin $\vec{0}$, the linear basis vectors $e_1,\dots,e_d$, and the horizon $h_\infty$, which determines a coordinate system. 
We can also express the coordinates of a point in projective space in terms of the operations meet, join, and cross ratio applied to the basis elements as follows. 
Let $\infty_i = (\vec{0} \vee e_i) \wedge h_\infty$ 
denote the point on the horizon in the direction of $e_i$, and let  
\[
h_{0,i} = e_0 \vee \dots \vee e_{i-1} \vee \vec{0} \vee e_{i+1} \vee \dots \vee e_d, \quad h_{1,i} = e_i \vee (h_{0,i} \wedge h_\infty) 
\]
denote the facet supporting hyperplanes of the unit cube. 
We project a point $p$ into these hyperplanes by 
$\pi_{i,x}(p):= \left(p \vee \infty_i \right)\wedge h_{i,x}$.
For $p \nin h_\infty$ the $i^\text{th}$ coordinate is given by 
\[(p)_i:=\left(p , \pi_{i,1}(p) | \pi_{i,0}(p), \infty_i \right).\]

\subsection{Tents} 

When one base of an prismoid is an edge, we call the prismoid a \df{tent}.  This edge is called the \df{apex} of the tent, and we refer to the other base face exclusively as the base of the tent. 
An abstract prismoid $\mc{T}$ is an \df{abstract tent} when one base is a combinatorial edge $\mc{A}$, the apex, and the sides of $\mc{T}$ satisfy the following.  A face of the base $f \in \mc{B}$ forms a side with the top face of the apex $\top \in \mc{A}$ if and only if the face $f$ forms a side with either both apex vertices or neither apex vertex,
\[
\mc{T} \subset \mc{B}\catprod \mc{A}, \quad 
\mc{A} := \{+, -, \top,\bot\}, \quad \bot \leq \pm \leq \top, 
\]
\[ 
(f,\top) \in \mc{T} \Leftrightarrow 
((f,+)\in\mc{T} \Leftrightarrow (f,-)\in\mc{T}). 
\]
An abstract tent is completely determined by its sides. 
We may therefore define an abstract tent from a function $\chi$ that takes a base face and returns a value indicating the set of sides that include the given base face.  For any $\chi:\mc{B}\to \{+,-,0,*\}$,  let 
\[ a(f) = 
\left\{\begin{array}{l@{\,}l@{\,}l@{\,}rl}
 \{ +, & & & \bot\} & \chi(f) = + \\
 \{ & -, & & \bot\} & \chi(f) = - \\
 \{ & & \top, & \bot \} & \chi(f) = 0 \\
 \{ +, & -, & \top, & \bot\} & \chi(f) = *, \\
\end{array}\right.  
\]
\[\tent(\chi):= \{ (f,a) :f \in\mc{B}, a \in a(f) \}. \] 
The set of all abstract tents with a given base $\mc{B}$ is the range of the function `$\tent$' over all functions $\chi$ on $\mc{B}$ with $\chi(\bot) = *$.


\begin{figure}[th]
\centering

\begin{tikzpicture}

\path (0,0) coordinate (o)
 (90:1) coordinate (b0)
 (270:1) coordinate (b1)
 (0:2.4) coordinate (E)
 (180:2.4) coordinate (W);

\foreach \i in {0,...,8}
{
 \path (\i*40:1.8) coordinate (a\i);
}
\foreach \i in {0,...,8}
{
 \path (\i*40+20:2.8) coordinate (A\i);
}

\matrix[column sep=1.5cm]
{

\draw[thick, fill=black!6] (a0) -- (a1) -- (a2) -- (a3) -- (a4) -- (a5) -- (a6) -- (a7) -- (a8) -- cycle;

\foreach \i in {0,...,3,5,6,7,8}
{
 \draw[teal,->] (o) -- (A\i);
}
\draw[teal,->] (W) -- (A4);
\draw[violet,very thick,->] (o) -- (E); 
\draw[violet,very thick,->] (o) -- (W);
\draw[thick] (b0) -- (b1); 
&

\begin{scope}
\clip (a5) ++(0,-2.7) rectangle (a0);
\shade[inner color=blue!50!black!30,outer color=white] (o) circle (2.7);
\end{scope}

\draw[thick,fill=white] (a0) -- (a1) -- (a2) -- (a3) -- (a4) -- (a5) -- (a6) -- (a7) -- (a8) -- cycle;

\draw[olive] 
(b0) -- (a0)
(b0) -- (a1)
(b0) -- (a2)
(b0) -- (a3)
(b0) -- (a4)
;
\draw[blue!50!black]
(b1) -- (a5)
(b1) -- (a6)
(b1) -- (a7)
(b1) -- (a8)
(b1) -- (a0)
;
\draw[thick] (b0) -- (b1); 

\\
};
\end{tikzpicture}

\smallskip

\begin{tikzpicture}[x={(.8cm,-.25cm)},y={(-1.2cm,-.15cm)},z={(0cm,.5cm)},scale=1.5]

\draw[thick,fill=black!6]
(0:1) coordinate (a0) -- (45:1) coordinate (a1) -- (90:1) coordinate (a2) -- (135:1) coordinate (a3) -- 
(180:1) coordinate (a4) -- (240:1) coordinate (a5) -- (270:1) coordinate (a6) -- (300:1) coordinate (a7) -- (330:1) coordinate (a8) -- cycle
;
\path
(intersection cs: first line={(a4)--(a5)}, second line={(a0)--(1,-1)}) coordinate (p)
    ++(115:2.2) +(0,0,2.2) coordinate (b1)
(p) ++(115:3.9) +(0,0,3.9) coordinate (b0)
;

\draw[thick]
(b0) -- (b1)
;

\draw[densely dotted]
($(p)!1.6!(a0)$) -- (p) -- ($(p)!1.6!(a5)$)
(p) -- (b1)
;

\foreach \i in {0,...,4}
{
 \draw[blue!50!black] (a\i) -- (b0);
}
\foreach \i in {5,...,8,0}
{
 \draw[olive] (a\i) -- (b1);
}

\end{tikzpicture}

\caption{
\textsc{Top Left:} The normal fans of a 9-gon and a segment. 
\textsc{Top Right:} The tent sides determined by the common refinement of the normal fans, with back visible edges of the 9-gon shaded. 
\textsc{Bottom:} A tent with the point $p \in \projcl (B) \wedge \projcl (A)$.}
\end{figure}
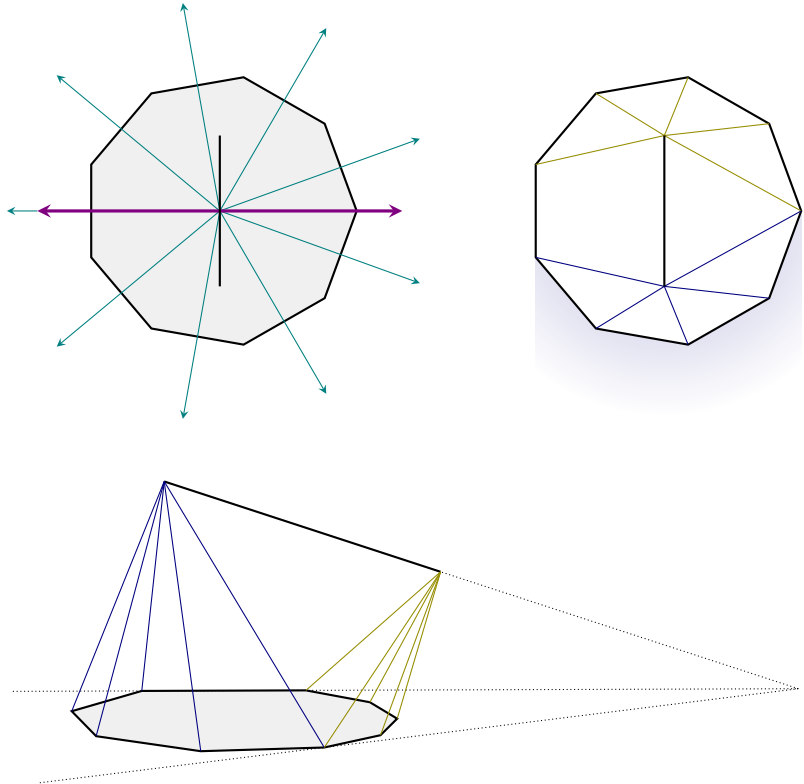

We now give the completion condition to a tent from its base. 
Specifically, this condition determines the visibility of the faces of the base $B$ from a point $p\nin B$. 
In this context, a tent may be called a Lawrence extension \cite{bayer1990lawrence}.  
An equivalent result appears in \cite[Section 3.3]{richter1996realization} 
along with further details of its history and use.

\begin{lem}\label{lem:tent} 
Given an abstract tent $\mc{T} = \tent(\chi)$ with base $\mc{B}$,
the completion condition for a realization $B$ of $\mc{B}$ to $\mc{T}$ from its base is that there be a projective copy $\tilde B$ of $B$ and a point $p$ such that  
for each face $f \in \mc{B}$ 
\[\vis(p,\tilde B,f) = \chi(f).\] 
Moreover, for a tent with base $B$ and apex $A$, $\projcl (B) \wedge \projcl (A) = \{p,p^-\}$ for $p \in \mb{O}^d$ satisfying the above.
\end{lem}

\begin{proof}
Observe that for a polytope $B \subset \mb{R}^d$, a face $f$ of $B$ is front visible from a point $p = \infty u$ with $u \in \mb{S}^{d-1}$ 
if and only if there is a normal vector $x \in \ncone(B,f)^\circ$ such that $\langle u,x \rangle >0$. 
The normal fan of a segment $A$ consists of a hyperplane $h = \ncone(A,\top)$ and the two half spaces it bounds.  Let $u \in \mb{S}^{d-1}$ be a normal vector of this hyperplane $h$.
Now $\nfan(A)$ is given by the sign of $x \mapsto \langle u, x \rangle$, and we may therefore index the vertices of $A$ by $\{+,-\}$ so that $\ncone(A,\pm) = \{x : \langle u, x \rangle = \pm\}$.  
Together this implies that the common refinement of $\nfan(B)$ and $\nfan(A)$ is determined by the visibility of faces of $B$ from $p = \infty u$.
Specifically, ${\ncone(B,f)^\circ \cap \ncone(A,+)^\circ} \neq \emptyset$ if and only if the face $f$ of $B$ is front visible from $p$, and likewise for $\ncone(A,-)$ and back visibility from $p$.
Thus by Lemma~\ref{lem:prism}, the completion condition for $B$ is $\vis(p,B,f) = \chi(f)$. 
Note that we may apply a small projective transformation to $(p,B)$ so that $\{p\} \cup B \subset \mb{R}^d$ without changing visibility. 
Furthermore, any realization of $\tent(\chi)$ is projectively equivalent to $B' \cv A' = [B;0] \cv [A;1]$, and assuming $B'$ is a facet, 
\[p' = \infty [u;0] = \infty(\face(A',+)-\face(A',-)) \in i_\mb{P}^{-1}(\projcl ([\mb{R}^d;0]) \wedge \projcl (A')) = i_\mb{P}^{-1}(\projcl (B')  \wedge \projcl (A')) \]
and $\vis(p',B',f) = \vis(p,B,f) = \chi(f)$. 
\end{proof}

\subsection{Transmitters} 

The purpose of a transmitter polytope is to impose a relationship between the the completion conditions of two of its faces.  The most restrictive example of this is the full transmitter $\mc{T}_\mc{B}$, which forces two faces to be projectively equivalent to each other.   
A slightly more general transmitter that will also be widely used is the forgetful transmitter 
$\mc{T}_{\mc{B}_0,\mc{B}_1}$.  The forgetful transmitter is similar to the full transmitter except one base $\mc{B}_0$ is a copy of the other base $\mc{B}_1$ where a set $\mc{V}$ of simple vertices are whittled.  
We can think of whittling a vertex as introducing linear constraints to a polytope that truncate that vertex.  The forgetful transmitter would force one base to be a projective copy of the other, but it ``forgets'' these extra linear constraints. 
The constructions of the full and forgetful transmitters given in this article are modified versions of those constructed in the proof of \cite[Theorem 5.1.1 and Theorem 5.3.1]{richter1996realization}, which
are very close to Lemma~\ref{lem:ftrans} below.\footnote{%
Note that `full transmitters' were simply called `transmitters' in \cite{richter1996realization}, but here `transmitter' will refer to a more general construction.}


\begin{figure}[h]

\begin{center}

\begin{tikzpicture}[yscale=-2/3, xscale=5/6]

 \filldraw[fill=orange!30!yellow!20] 
 (-75:1.8 and .6) coordinate (a1)
 -- (-130:2 and .6) coordinate (a2)
 -- (-170:2 and .6) coordinate (a3)
 -- (-200:2 and .6) coordinate (a4)
 -- (-310:2 and .6) coordinate (a5)
 -- cycle;
 \filldraw[fill=blue!20]   
  (0,4) +(-75:3.6 and 1.2) coordinate (b1)
 -- +(-130:4 and 1.2)  coordinate (b2)
 -- +(-170:4 and 1.2)  coordinate (b3)
 -- +(-200:4 and 1.2)  coordinate (b4)
 -- +(-310:4 and 1.2)  coordinate (b5)
 -- cycle;
 \foreach \i in {1,...,5}
 {
  \draw (a\i) -- (b\i);
  \draw[densely dotted] (a\i) -- (0,-4);
  \draw[black!75] (a\i) -- (0,1);
  \draw[black!75] (b\i) -- (0,2.5);
 }
 \draw[black!75] (0,1) -- (0,2.5);
 \draw[black!75, densely dotted] (0,-4) -- (0,1);

\end{tikzpicture}
\hspace{24pt}
\begin{tikzpicture}[yscale=-2/3,xscale=5/6]

 \path
  (0,4) +(-75:3.6 and 1.2) coordinate (b1)
  +(-130:4 and 1.2)  coordinate (b2)
  +(-170:4 and 1.2)  coordinate (b3)
  +(-200:4 and 1.2)  coordinate (b4)
  +(-310:4 and 1.2)  coordinate (b5)
 ($(b5)!.25!(b4)$) coordinate (c1)
 ($(b5)!.5!(b1)$) coordinate (c2)
  (-75:1.8 and .6) coordinate (a1)
  (-130:2 and .6) coordinate (a2)
  (-170:2 and .6) coordinate (a3)
  (-200:2 and .6) coordinate (a4)
  (-310:2 and .6) coordinate (a5)
;
 
 \filldraw[fill=orange!30!yellow!20] (a1) -- (a2) -- (a3) -- (a4) -- (a5) -- cycle;
 \filldraw[fill=blue!20] (b1) -- (b2) -- (b3) -- (b4) -- (c1) -- (c2) -- cycle;
 \draw (c1) -- (c2) -- (a5) -- cycle; 
 \foreach \i in {1,...,4}
 {
  \draw (a\i) -- (b\i);
  \draw[densely dotted] (a\i) -- (0,-4);
  \draw[black!75] (a\i) -- (0,1);
  \draw[black!75] (b\i) -- (0,2.5);
 }
 \draw[densely dotted] (b5) -- (0,-4) 
 (c1) -- (b5) -- (c2) ;
 \draw[black!75] (0,1) -- (0,2.5) -- (a5) -- cycle
 (c1) -- (0,2.5)
 (c2) -- (0,2.5);
 \draw[black!75, densely dotted] (0,-4) -- (0,1);

\end{tikzpicture}

\caption{
Schlegel diagrams of a full transmitter (\textsc{Left}) and forgetful transmitter (\textsc{Right}).}\label{fig:trans}
\end{center}
\end{figure}
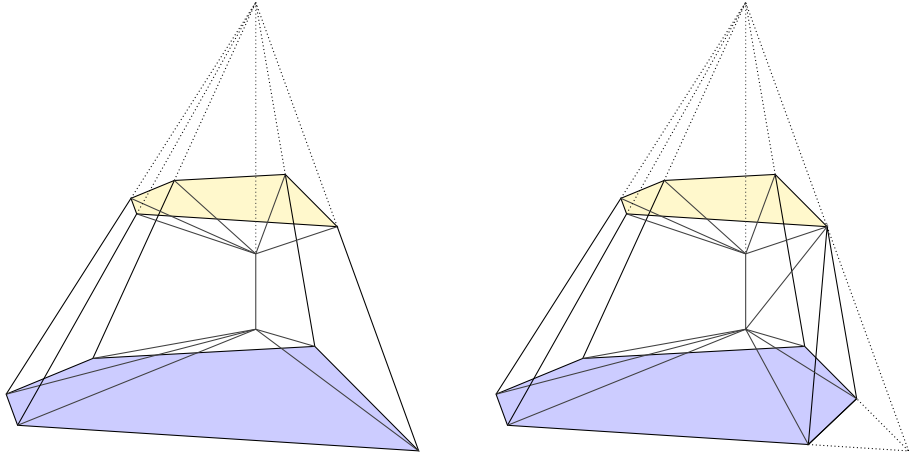

Before defining transmitters, we first state the main lemma for full and forgetful transmitters.  We then give the construction and the main lemma for general transmitters, Lemma~\ref{lem:projtrans}.  The proof of Lemma~\ref{lem:ftrans} will then follow as a special case.

\begin{lem}\label{lem:ftrans}   
Given combinatorial polytopes $\mc{B}_0, \mc{B}_1$ where $\mc{B}_0$ is a copy of $\mc{B}_1$ whittled at vertices $\mc{V}$,
the completion condition for realizations $B_i$ of $\mc{B}_i$ to a forgetful transmitter $\mc{T}_{\mc{B}_0, \mc{B}_1}$ from its bases is that there be 
projective copies $\tilde B_i$ of $B_i$ 
such that for each face $f \in \mc{B}_1 \setminus \mc{V}$ 
\[ \face(\tilde B_0,f) \subset \face(\tilde B_1,f). \] 
In particular for the case $\mc{V} = \emptyset$, the completion condition to a full transmitter from its bases is that the polytopes $B_i$ be projectively equivalent, ${B_0\projeq B_1}$.
\end{lem}

A \df{transmitter} is any polytope that is combinatorially equivalent to one of the form 
\[ \trmr(B_0,B_1) :=  [B_0 ; 0; 0] \cv [B_1; 1; 0] \cv [\vec{0}; 0 ; 1]  \cv [\vec{0}; 1 ; 1] \]
for polytopes $B_0, B_1 \subset \mb{R}^d$. 
An \df{abstract transmitter} is any poset $\mc{T} = \trmr(\mc{P},\mc{F}_0,\mc{F}_1)$ of the following form. 
For an abstract prismoid $\mc{P} \subset \mc{B}_0 \catprod \mc{B}_1$, and a subset $\mc{F}_i \subset \mc{B}_i$ of each base such that $\top \nin \mc{F}_i $ and $\bot \in \mc{F}_i$, let 
\[ \trmr(\mc{P},\mc{F}_0,\mc{F}_1) := \tent(\chi) \quad \text{where} \]
\[  \chi(f_0,f_1) = \left\{\begin{array}{cl} 
- & f_0 \in \mc{F}_0,\ f_1 \nin \mc{F}_1 \\
+ & f_0 \nin \mc{F}_0,\ f_1 \in \mc{F}_1 \\ 
0 & f_0 \nin \mc{F}_0,\ f_1 \nin \mc{F}_1 \\ 
* & f_0 \in \mc{F}_0,\ f_1 \in \mc{F}_1. \\ 
\end{array}\right.  \]
We call $\mc{P}$ the prismoid of the transmitter, and we say that the bases of the prismoid, $\mc{B}_0,\mc{B}_1$, are the bases of the transmitter. 
We refer to sides of $\mc{T}$ as transmitter sides and the sides of $\mc{P}$ as prismoid sides. 
Note that we may equivalently define a transmitter as a realization of an abstract transmitter.

The pair of sets $\mc{F}_i$ tell us about the visibility of faces of the transmitter $\mc{T} = \trmr(\mc{P},\mc{F}_0,\mc{F}_1)$. 
Specifically by Lemma~\ref{lem:tent}, for any realization of $\mc{T}$ there is 
a point $p$ such that one base $B_0$ is front only visible from $p$ and the other base $B_1$ is back only visible from $p$, and likewise for the faces of the bases.  If a face $f_0$ is in $\mc{F}_0$, then $f_0$ is doubly visible from $p$, else $f_0$ is front only visible.  Similarly, if a face $f_1$ is in $\mc{F}_1$, then $f_1$ is doubly visible, else $f_1$ is back only visible.

The main lemma for transmitters, Lemma~\ref{lem:projtrans}, gives the completion condition from the bases in terms of strictly supporting half-spaces.  A half-space $H$ \df{strictly supports} a polytope $B$ at a specified face $f$ when the polytope is contained in the half-space and intersects the boundary of the half-space in the specified face, $B \subset H$ and $B \cap \partial H = \face(B,f)$.

\vbox{
\begin{lem}\label{lem:projtrans}
Given an abstract transmitter $\mc{T} = \trmr(\mc{P},\mc{F}_0,\mc{F}_1)$ with bases $\mc{B}_0, \mc{B}_1$, 
the completion condition for realizations $B_i$ of $\mc{B}_i$ to $\mc{T}$ from its bases are that there be projective copies $\tilde B_i$ of $B_i$ such that the following holds:
\begin{enumerate}[label=T\arabic*]
\item\label{condTP} $\tilde B_0, \tilde B_1$ satisfy the condition of Lemma~\ref{lem:prism}: the common refinement of the normal fans of the $\tilde B_i$ is dual to the sides of $\mc{P}$,
\[ \labl(\nfan(\tilde B_0) \wedge \nfan(\tilde B_1))^* = \side(\mc{P})  .\]
\item\label{condTB0} For each face $f \in \mc{B}_0$, there is a half-space $H_f$ that contains $\tilde B_1$ in its interior and that strictly supports $\tilde B_0$ at $f$ if and only if $f \in \mc{F}_0$. 
\item\label{condTB1} For each face $f \in \mc{B}_1$, there is a half-space $H_f$ that contains $\tilde B_0$ in its interior and that strictly supports $\tilde B_1$ at $f$ if and only if $f \in \mc{F}_1$. 
\end{enumerate}
\end{lem}
}\nointerlineskip

\begin{proof}
We start by relating parts \ref{condTB0} and \ref{condTB1} of the condition of the theorem to the visibility of a face from a point. 
Consider a realization $P = [\tilde B_0 ; 0] \cv [\tilde B_1; 1]$ of $\mc{P}$,
and a face $f \in \mc{B}_0$. Let $F = \face(\tilde B_0,f)$, 
let $X = \bigcup_{t\geq 0}[F;t]$ be the set of points above $[F;0] = \face(P,(f,\bot))$ in the last coordinate,  
and let $q = \infty[\vec{0};1]$. 
Observe that ${X} \cup q = [F;0] \cv q$.
We claim the face $(f,\bot)$ of $P$ is visible from $q$ if and only if there is a half-space $H$ that contains $\tilde B_1$ in its interior and that strictly supports $\tilde B_0$ at $f$.

First suppose that the face $(f,\bot)$ of $P$ is not visible from $q$.  Then, $X^\circ$ intersects $P$. 
Suppose also that there is a half-space $H$ that strictly supports $\tilde B_0$ at $f$ and contains $\tilde B_1$ in its interior.  Then, the hyperplane $h' = [\partial H; \mb{R}]$ intersects $P$ in $[F;0]$ and contains $X$, so $X \cap P = [F;0]$.  Since $[F;0]$ is on the boundary of $X$, this contradicts $X^\circ \cap P \neq \emptyset$.  Thus, no such half-space exists.

Now suppose that the face $(f,\bot)$ of $P$ is visible from $q$.  Then, $X^\circ$ is disjoint from $P$. 
We will show that the affine span of $X$ intersects $P$ in $[F;0]$.  Suppose not and let $y \in P \setminus [F;0]$ be a point in the the affine span of $X$ and let $x \in F^\circ$.  Since the restriction of the affine span of $X$ to $[\mb{R}^d;0]$ is the affine span of $[F;0]$, which intersects $P$ in $[F;0] \not\ni y$, the last coordinate of $y$ must be positive, so $y = [z;t]$ for some $z$ in the affine span of $F$ and $t > 0$.  This implies that for $\eps > 0$ sufficiently small $\eps z +(1-\eps)x \in F^\circ$, so $\eps y +(1-\eps)[x;0] \in X^\circ \cap P$.  But this contradicts $X^\circ \cap P = \emptyset$, so the affine span of $X$ does indeed intersect $P$ in $[F;0]$.  The affine span of $X$ can be extended to a hyperplane of the form $h' = [h;\mb{R}]$ such that $h' \cap P = [F;0]$, and so $h$ bounds a half-space of the form $H' = [H;\mb{R}]$ that strictly supports $P$ at $(f,\bot)$.  Therefore, the half-space $H$ contains $\tilde B_1$ in its interior and strictly supports $\tilde B_0$ at $f$.
Hence, the claim holds.

Since we can exchange $\tilde B_0$ and $\tilde B_1$ by symmetry in the above claim, we have the following. 
If $P = [\tilde B_0 ; 0] \cv [\tilde B_1; 1]$ realizes $\mc{P}$ and $q = \infty[\vec{0};1]$,
then $\vis(q,P,(f_0,f_1)) = \chi(f_0,f_1)$ for $\chi$ as in the definition of a transmitter if and only if parts \ref{condTB0} and \ref{condTB1} hold.

Next, we show that the condition of theorem is sufficient. 
Suppose we are given a pair $\tilde B_0, \tilde B_1$ that satisfy the condition. 
By Lemma~\ref{lem:prism} $P = [\tilde B_0 ; 0] \cv [\tilde B_1; 1]$ is a realization of $\mc{P}$. 
By the above claim $\vis(q,P,(f_0,f_1)) = \chi(f_0,f_1)$, so by Lemma~\ref{lem:tent}, $[P;0] \cv [\vec{0}; 0 ; 1]  \cv [\vec{0}; 1 ; 1]$ is a realization of $\mc{T}$. 

Finally, we show that the condition of theorem is necessary. 
Suppose we are given a realization $T$ of $\mc{T}$.  By Lemma~\ref{lem:prism}, there is some projective copy of $T$ of the form
\[ \tilde T = [P;0] \cv [A;1], \quad P = [\tilde B_0;0] \cv [\tilde B_1;1], \quad A = \conv\{a_0, a_1\}, \]
and part \ref{condTP} holds. 
Since the faces $\tilde B_0$ and $\tilde B_1$ are respectively back only and front only visible from a point $q \in \projcl(P) \wedge \projcl(A)$, we may assume that $a_i = [\vec 0;i]$ and $q = \infty[\vec{0};1;0]$. 
Otherwise, if we first replace $a_i$ with $a'_i = {(a_0 \vee a_1) \cap [\mb{R}^d;i;\mb{R}]}$ and then translate $a'_i$ and $\tilde B_i$ by $-a'_i$, we obtain an affine copy of $\tilde T$ as above in which the $a_i$ have become $[\vec 0;i]$.
By Lemma~\ref{lem:tent}, we have $\vis(q,P,(f_0,f_1)) = \chi(f_0,f_1)$, which implies by the above claim that parts \ref{condTB0} and \ref{condTB1} hold.
\end{proof}

An \df{abstract full transmitter} with base $\mc{B}$ is a poset of the form 
\[ \mc{T}_\mc{B}:=\trmr(\pris(\mc{B}),\emptyset, \emptyset) \quad \text{where} \] 
\[ \pris(\mc{B}) := (\mc{B}\times \bot) \cup (\bot\times \mc{B}) \cup \{(f,f) : f \in \mc{B}\} \]  
and $\mc{B}$ is a bounded poset. 
Observe that the prismoid of $\mc{T}$ is a prism $\mc{P}$, and for any realization $T \in \mb{R}^d$ of $\mc{T}_\mc{B}$, all sides of the prismoid of $T$ are doubly obscured from some point $p \in {\mb{S}}{}^d$. 
Intuitively, one base is the shadow cast by the other base from a light source at the point $p$; see Figure~\ref{fig:trans} (left).
An \df{abstract forgetful transmitter} with bases $\mc{B}_i$ is a poset $\mc{T}_{\mc{B}_0, \mc{B}_1}$ of the following form. 
Let $\mc{B}_1$, $\mc{W}_1,\dots$, $\mc{W}_n$ be a bounded poset of rank $d{+}1$, and $\{v_1,\dots,v_n\} = \mc{V} \subset \mc{B}_1$ be a set of simple vertices (atoms covered by $d$ elements each), and $\varphi_i : [v_i,\top] \to \mc{W}_i$ be monotonic functions such that $\varphi_i([v_1,\top]) = [w_i,\top]$ where $w_i = \varphi_i(v_1)$, 
and let $\mc{B}_0 = \mc{B}_1 \#^*_{\varphi_1} \mc{W}_1 \dots  \#^*_{\varphi_n} \mc{W}_n $.
\[ \mc{T}_{\mc{B}_0, \mc{B}_1} := \trmr(\mc{P}, \emptyset, \mc{V}) \quad \text{where} \]
\[ \mc{P} = (\mc{B}_0 \times \bot) \cup (\bot \times \mc{B}_1)
 \cup \{ (f,f): f \in \mc{B}_1 \setminus \mc{V}\} 
\cup \bigcup_{i=1}^n ((\mc{W}_i \setminus [w_i,\top]) \times \{v_i\} ). 
\]
A full or forgetful transmitter is any realization of the respective posets above.

\begin{proof}[Proof of Lemma \ref{lem:ftrans}] 
We start with the completion condition for full transmitters.  Since, $\mc{F}_0 = \mc{F}_1 = \emptyset$ in the condition of Lemma~\ref{lem:projtrans}, and a polytope is defined as the intersection of the half-spaces of its affine span strictly supporting its facets, the completion condition is $\tilde B_0 = \tilde B_1$. 

For forgetful transmitters $\mc{T}_{\mc{B}_0,\mc{B}_1}$, 
since $\mc{F}_0 = \emptyset$, part of the completion condition is $\tilde B_0 \subset \tilde B_1$.  Since  $\mc{F}_1 = \mc{V}$, part of the condition is $\face(\tilde B_1, v_i) \nin \tilde B_0$ and for all $f \in \mc{B}_1 \setminus \mc{V}$, $\face(\tilde B_1, f) \cap \tilde B_0 \neq \emptyset$.  And, since $(f,f)$ is a side of $\mc{T}_{\mc{B}_0,\mc{B}_1}$, part of the condition is $\face(\tilde B_1, f)$ is parallel to $\face(\tilde B_0, f)$.  All together, these are equivalent to the condition $\face(\tilde B_0, f) \subset \face(\tilde B_1, f)$ for all $f \in \mc{B}_1 \setminus \mc{V}$.
Note that the rest of the completion condition of Lemma~\ref{lem:projtrans} is a consequence of the above condition.  Namely, 
$\ncone(\tilde B_0,f_0) \subset \ncone(\tilde B_1, v_i)$ for $f_0 \in \mc{W}_i \setminus [w_i,\top]$.
\end{proof}

\subsection{Combining Completion Conditions} 

Recall that a full transmitter has two facets that are pyramids, and these facets are always projectively equivalent.  Suppose we are given a polytope with a specified facet that is a pyramid, and we are given a completion condition for this facet.  Then, if we glue a transmitter along this specified facet, we get a new polytope with the same completion condition from the transmitter's one remaining pyramidal facet.  So far this does not give us anything new, but using connectors instead of transmitters will give us more.  
Connectors serve the same purpose as full transmitters, but have more than just two pyramidal facets that are projectively equivalent.  
Suppose now we are given several polytopes and completion conditions $\psi_1,\dots,\psi_n$ from a specified pyramidal facet of each polytope. 
If we glue the specified facets of these polytopes along the pyramidal facets of a connector then the completion condition for any the connector's remaining pyramidal facets will be the conjunction of the given completion conditions $\psi_1 \wedge \dots \wedge \psi_n$. 

We now define the \df{connector} $\conn(n,\mc{B})$ with $n$ facets of type $\pyr(\mc{B}) := \mc{B} \catprod \{\bot,\top\}$.
For $n=2$, this is just the full transmitter, and 
for $n=4$, this consists of two copies of the full transmitter glued together along the prism of each 
\[ \conn(2,\mc{B}):=\trmr(\mc{P},\emptyset,\emptyset) \quad \text{where} \quad
\mc{P} = \pris(\mc{B}) := \{(f,\bot),(\bot,f),(f,f): f \in \mc{B}\} \subset \mc{B} \catprod \mc{B}, \]  
\[ \conn(4,\mc{B}):=\conn(2,\mc{B}) \#_{\varphi}  \conn(2,\mc{B}) \]
where $\varphi : (\mc{P} \times \bot) \to (\mc{P} \times \bot) $ is the identify function; 
see Figure~\ref{fig:connector}.  
For $n>4$ even, this consists of $\frac{n}{2}-1$ copies of $\conn(4,\mc{B})$ glued together along their pyramids, and for $n$ odd we just disregard a pyramid of $\conn(n+1,\mc{B})$. 
The choices of pyramids to glue along or disregard is irrelevant.
Generally, the value $n$ will be the number of other polytopes that are glued to a connector, so we denote $\conn(n,\mc{B})$ simply by $\conn(\mc{B})$ in this case. 
This generalizes the connector polytope in \cite[Section 5.2]{richter1996realization}.


\begin{figure}[th]
\centering

\begin{tikzpicture}

\begin{scope}[scale=1]
 \path (-.4,0) coordinate (a)
 +(.8,-1) coordinate (b)
 +(.8,1)  coordinate (c)
 +(0,2) coordinate (d) 
 +(-.6,.2) coordinate (e) 
 +(-.6,.8) coordinate (f) 
 +(1.4,.2) coordinate (g)
 +(1.4,.8) coordinate (h);
\end{scope}

\matrix[row sep=6mm]
{

\begin{scope}[scale=2/5]
 \path
 +(190:6 and 2.4) coordinate (a1)
 -- +(250:6 and 2.4) coordinate (a2)
 -- +(335:6 and 2.4) coordinate (a3)
 -- +(430:6 and 2.4) coordinate (a4)
 -- +(475:6 and 2.4) coordinate (a5)
 -- cycle;
 \path
 (0,4.2) +(190:1.5 and .6) coordinate (b1)
 -- +(250:1.5 and .6)  coordinate (b2)
 -- +(335:1.5 and .6)  coordinate (b3)
 -- +(430:1.5 and .6)  coordinate (b4)
 -- +(475:1.5 and .6)  coordinate (b5)
 -- cycle;
 \path (0,2.7) coordinate (c)
 (0,3.1) coordinate (d)
 (0,7) coordinate (e)
 (0,9) coordinate (f)
 (0,5.6) coordinate (g);
\end{scope}

 \filldraw[draw=orange!40!black,fill=orange!40!yellow!30] 
  +(a1) -- +(a2) -- +(a3) -- +(a4) -- +(a5) -- cycle;
 \filldraw[draw=blue!30!black,fill=blue!20] 
  +(b1) -- +(b2) -- +(b3) -- +(b4) -- +(b5) -- cycle;
 \foreach \i in {1,...,5}
 {
  \draw (a\i) -- (b\i);
  \draw[black!60!violet] (b\i) -- (e);
  \draw[black!60!red] (a\i) -- (f);
 }
 \draw[black!75] 
 (e) -- (f);
\\ 

 \filldraw[draw=orange!40!black,fill=orange!40!yellow!30] 
  +(a1) -- +(a2) -- +(a3) -- +(a4) -- +(a5) -- cycle;
 \filldraw[draw=blue!30!black,fill=blue!20] 
  +(b1) -- +(b2) -- +(b3) -- +(b4) -- +(b5) -- cycle;
 \foreach \i in {1,...,5}
 {
  \draw  +(a\i) -- +(b\i);
  \draw[olive]  +(a\i) -- +(c);
  \draw[black!70!green]  +(b\i) -- +(d);
 }
 \draw[black!75]  +(c) -- +(d);
\\
};
\end{tikzpicture}
\begin{tikzpicture}
\node[scale=1.6] at (0,0) {
\begin{tikzpicture}

 \filldraw[draw=orange!40!black,fill=orange!40!yellow!30] 
  +(a1) -- +(a2) -- +(a3) -- +(a4) -- +(a5) -- cycle;
 \filldraw[draw=blue!30!black,fill=blue!20]  
  +(b1) -- +(b2) -- +(b3) -- +(b4) -- +(b5) -- cycle; 
 \foreach \i in {1,...,5}
 {
  \draw (a\i) -- (b\i);
  \draw[olive] (a\i) -- (c);
  \draw[black!70!green] (b\i) -- (d);
  \draw[black!60!violet] (b\i) -- (e);
  \draw[black!60!red] (a\i) -- (f);
 }
 \draw[black!75] (c) -- (d)
 (e) -- (f);

\end{tikzpicture}
};
\end{tikzpicture}

\caption{
\textsc{Left:} Two Schlegel diagrams of full a transmitter. 
\textsc{Right:} A Schlegel diagram of the connector with 4 pyramids that results from gluing these transmitters. (Compare with Figure~\ref{fig:glue} and Figure~\ref{fig:trans} Left)
}\label{fig:connector}

\end{figure}
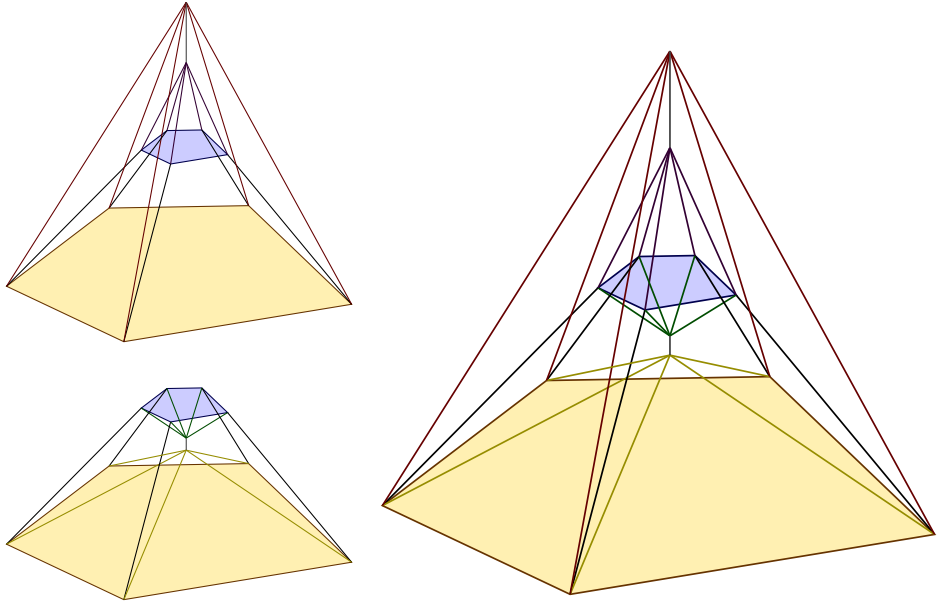

\begin{lem}\label{lem:cons} 
Given a combinatorial polytope $\mc{B}$, 
the completion condition for realizations $P_1,\dots,P_n$ of $\pyr(\mc B)$ to a connector polytope $\conn(n,\mc{B})$ from its $n$ pyramidal facets are that all realizations be projectively equivalent $P_1 \projeq \dots \projeq P_n$.
\end{lem}

\begin{proof}
First recall that two pyramids are projectively equivalent if and only if their bases are projectively equivalent.  
For $n=2$, the condition is the same as Lemma \ref{lem:ftrans} for full transmitters applied to the bases of the pyramids.  By Lemma \ref{lem:glue}, identical copies of a realization of $\conn(2,\mc{B})$ can always be glued together as in the definition of $\conn(n,\mc{B})$, so the condition is sufficient. 
By Lemma \ref{lem:unglue}, any realization of $\conn(4,\mc{B})$ is the union of two realizations of $\conn(2,\mc{B})$ intersecting along the prism of each, since prisms are necessarily flat, so the condition is necessary for $n=4$, or 3.  Likewise for $n>4$, any realization of $\conn(n,\mc{B})$ is the union of a realization of $\conn(4,\mc{B})$ and of $\conn(n-2,\mc{B})$ intersecting along a pyramid of each, and all pyramids of both pieces are projectively equivalent to the pyramid where they intersect, so the condition is necessary in this case. 
\end{proof}


While gluing connectors and transmitters together, we may in some cases want to glue along a lower dimensional face instead of a facet.  For this we can repeatedly stellate a facet containing this face until we have a facet that is an iterated pyramid over this face. We can then glue along this iterated pyramidal facet instead.  These repeated stellations are equivalent to gluing a single polytope to the facet, which we call an \df{adapter} \cite[p.\ 149]{richter1996realization}. 
We define the adapter inductively. 
\[ 
\adapt(\mc{F}, f) := \pyr(\mc{F}) \quad 
\text{for a facet (coatom) } f \in \mc{F}, 
\]
\[
\adapt(\mc{F}, g) := \pyr(\mc{F}) \#_\varphi \pyr(\adapt([\bot,f],g)) \quad 
\text{for } g\in \mc{F}
\]
where $f \geq g$ is a facet of $\mc{F}$ and 
$\varphi : [\bot,(f,\top)] \to [\bot, ((f,\bot),\top)]$, $\varphi(x,y) = ((x,\bot),y)$.
Note that this depends on choosing a chain of faces $g \prec \dots \prec f$, but this choice of is irrelevant, so we leave it implicit. 
We glue the adapter along the facets that have indices $a_\mc{F} = (\mc{F},\bot)$ and $a_g = (\dots(g,\top)\dots,\top)$. 
Note that the facet $a_g$ is an iterated pyramid over the face $g \in \mc{F}$, and that the base $b_g$ of the facet $a_g$ has index $(\dots(g,\bot)\dots,\bot)$ and is identified with the face $(g,\bot) < a_\mc{F}$ by the sequence of maps $\varphi$ in the inductive definition of the adapter. For a polytope $B$, let $\pyr(B) := [B;0] \cv [\vec{0};1]$. 

\begin{lem}\label{lem:adapt} 
Given a combinatorial polytope $\mc{F}$ and a face $g \in \mc{F}$, 
the completion condition for $F$, $G$ to $\mc{A} = \adapt(\mc{F},g)$ from facets $a_\mc{F}$, $a_g$ is 
that $G$ be projectively equivalent to $\pyr \circ \dots \circ \pyr \circ \face(F,g)$. 
\end{lem}

\begin{proof}
This is immediate by construction.  For a realization $A$ of $\mc{A}$, 
$\face(A,b_g) = \face(\face(A,a_\mc{F}),b_g)$ and $\face(A,a_g) = \pyr \circ \dots \circ \pyr \circ \face(A,b_g)$. 
\end{proof}

\subsection{Stamp of the Cube} 

We now have a way to combine the completion conditions of several polytopes' faces in a single polytope.  As a simple example, we construct our first stamp, the stamp of the unit $d$-cube.  In the general stamp construction, we will use the stamp of the cube as a scaffolding to which we fix points. 
We denote the unit $d$-cube by $\cube^d := [0,1]^d$; we may omit $d$ for brevity.   
Let $\labl(\cube) = \{0,1,*\}^d \cup \{\bot\} $ where 
\[ \face(\cube,c) = \left\{x \in \cube : x_i = c_i \text{ if } c_i \in \{0,1\} \right\}. \]

The \df{unit cube stamp} $\mc{S}_\subcube = \mc{S}_{\subcube^d}$ is constructed by gluing posets defined in the above sections together as illustrate in Figure~\ref{glue:cubestamp}. 
$\mc{S}_{\subcube}$ consists of a connector $\mc{C}_{\subcube} = \conn(d\!+\!1,\cube)$ glued to $\pyr(\mc{T}_i)$ for each $i=1,\dots,d$ where $\mc{T}_i = \tent(\cube,\chi_i)$ and
\[\chi_i(c) = \left\{\begin{array}{ll}
+ & c_i = 1 \\
- & c_i = 0 \\
0 & c_i = *. \\
\end{array}\right.\]
Specifically, one of the connector's pyramidal facets is glued to $\pyr(\mc{T}_i)$ along the pyramid over the base of $\mc{T}_i$.  Note that $\mc{T}_i$ is an example of a full transmitter, but is not used in that capacity here. 
One pyramidal facet of the connector remains unglued. 
Let $f_\subcube$ be the base of this unglued pyramidal facet of $\mc{C}_{\subcube}$ in $\mc{S}_{\subcube}$. 


\begin{figure}[h!]
\begin{center}

\begin{tikzpicture}
 \node at (0,0) [draw] (cc) {$\mc{C}_{\subcube^d}$};
 \node at (1.3,-1.1) [draw] (twi) {$\pyr(\mc{T}_i)$} ;
 \node at (3.2,-1.1) {$ i=1, \cdots, d$};
	\node at (-2,0) [draw] (u) {$\mc{S}_{\subcube^d}$};
	\node at (-1,0) {\scalebox{1.5}{$:=$}};
 
 \draw (0,.5) -- (cc) -- (3.7,0)
  (1.3,0) -- (twi)
  (2.3,0) -- +(0,-.5)
  (3.3,0) -- +(0,-.5)
  (u) -- (-2,.5);

\end{tikzpicture}

\caption{Gluing diagram for the stamp the unit cube.}\label{glue:cubestamp}
\end{center}
\end{figure}
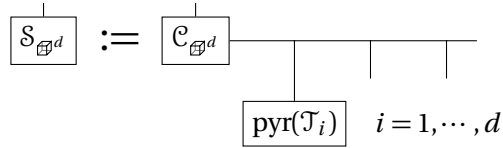

\begin{lem}\label{lem:unit}
The completion condition for $F$ to $\mc{S}_{\subcube^d}$ from the facet $f_{\subcube}$ is that $F$ be projectively equivalent to the unit $d$-cube, $F \projeq \cube$. 
\end{lem}

For the proof we will make use of the fact that the unit cube determines a coordinate system for an affine part of real projective space. 
Recall that a choice of origin $\vec{0}$, linear basis vectors $e_i$, and horizon $h_\infty$ form a basis in a projective space,
which determines a coordinate system.

\begin{proof}[Proof of Lemma \ref{lem:unit}]
In the construction of $\mc{S}_{\subcube}$ we only glue along pyramids, which are necessarily flat, so
by Lemmas~\ref{lem:glue} and~\ref{lem:unglue}, $F$ can be completed from $f_{\subcube}$ to $\mc{S}_{\subcube}$ if and only if we can realize each of the facets that we glue along such that they satisfy the completion condition of each of the pieces and $F$ is the realization of the specified face $f_{\subcube}$.  By Lemma~\ref{lem:cons}, this is equivalent to $F$ satisfying the completion condition to each transmitter $\mc{T}_i$ from its prismoid. 

We claim that the combined completion conditions for $F \combeq \cube$ from the prismoid to $\mc{T}_i$ for $i=1,\dots,d$ are equivalent to $F \projeq \cube$. 
Denote the indices of the facets of $\cube$ by $f_{i,x} = c$ for $x =0,1$ where $c_i = x$ and $c_j = *$ for $j \neq i$. 
By Lemma~\ref{lem:tent}, the completion condition to $\mc{T}_i$ is the following.  The set of lines spanned by an edge connecting the facets $f_{i,0}$ and $f_{i,1}$ of $F$ all meet at a common point; let $\infty_i$ denote this point. 
For one direction, $F \projeq \cube$ implies these conditions, since the unit cube satisfies these conditions.
Given a realization $S$ of $\mc{S}_\subcube$ with $F = \face(S,f_{\subcube})$, we use $F$ to define a coordinate system.  
Choose the vertex where the facets $f_{1,0},\dots,f_{d,0}$ of $F$ meet to be the origin $\vec{0}$, and each of its neighboring vertices to be the standard basis vectors $e_1,\dots, e_d$, and $h_\infty:=\infty_1 \vee \dots \vee \infty_d$ to be the horizon.    This gives 
$F_{i,0} = \face(S,f_{i,0}) \subset {\{x : (x)_i = 0\}}$, since $F_{i,0}$ contains $\vec{0}$ and $e_j$ for $j\neq i$,
and gives 
$F_{i,1} = \face(S,f_{i,1}) \subset {\{x : (x)_i = 1\}}$, since $F_{i,1}$ contains $e_i$ and $\infty_j$ for $j\neq i$.  
Hence $F$ is the unit cube in this coordinate system,  
and therefore the conditions imply $F \projeq \cube$. 
\end{proof}

\subsection{Lamppost Polytopes} 

A lamppost polytope with base $B$ depends on a pair of faces of $B$ and a visibility function $\chi$.  
Like a tent, a lamppost polytope determines the visibility of the faces of its base from a certain point $p$.
We may think of $p$ as a light-source that illuminates these faces, ``the lamp.''  A lamppost polytope additionally requires $p$ to be on a line $\ell$ through the specified pair of faces, ``the lamppost;'' see Figure~\ref{fig:lamppost}.  
Formally, for a bounded poset $\mc{B}$, a function ${\chi:\mc{B} \to \{+,-,0,*\}}$, and a pair of incomparable faces $f_0, f_1 \in \mc{B}$, let 
\[ \lamp(\chi,f_0,f_1) := \trmr(\tent(\chi),\mc{B}\setminus([f_0,\top]\cup[f_1,\top]),\emptyset). \]
An \df{abstract lamppost polytope} with base $\mc{B}$ is a poset of the form $\lamp(\chi,f_0,f_1)$, and a realization is called a lamppost polytope. 
This generalizes the marvelous {``polytope~$X$''} from \cite[Section 5.4]{richter1996realization}. 

\begin{figure}[ht]
\centering

\begin{tikzpicture}[scale=1.2]

\begin{scope}[rotate=90]

\foreach \i in {0,1,...,7}
{ \path (45*\i:1) coordinate (a\i); }

\draw[ultra thick]
(a1) -- (a2)
(a7) -- (a6)
;

\coordinate (p) at (intersection cs: first line={(a0)--(a4)}, second line={(a1)--(a2)});
\coordinate (b1) at ($(a1)!4!(a2)$);
\coordinate (b2) at ($(a7)!4!(a6)$);

\draw[thin,yellow,fill=yellow!10]
(a1) -- (p) -- (a7) -- (a0) -- cycle
;

\begin{scope}
\clip
(b1) -- (a2) -- (a3) -- (a4) -- (a5) -- (a6) -- (b2) -- cycle;
\shade[inner color=blue!70!black!50,outer color=white] (a0) circle (2.4cm);
\end{scope}

\draw[thick,fill=black!6]
(a7)
\foreach \i in {0,1,...,7}
{ -- (a\i) }
;

\draw
(p) -- (-1.6,0)
;
\fill[yellow!80!orange] (p) circle (.09cm);

\end{scope}

\begin{scope}[x={(-.8cm,-.6cm)},y={(-.36cm,.48cm)},z={(.48cm,-.64cm)},shift={(4.5cm,0cm)}]

\foreach \i in {0,1} {
\foreach \j in {0,1} {
\foreach \k in {0,1} {
\coordinate (a\i_\j_\k) at ({1-2*\i},{1-2*\j},{1-2*\k});
}}}

\foreach \i in {0,1} {
\foreach \j in {0,1} {
\coordinate (b\i_\j) at (0,{.618-1.236*\i},{1.618-3.236*\j});
}}

\foreach \i in {0,1} {
\foreach \j in {0,1} {
\coordinate (c\i_\j) at ({.618-1.236*\i},{1.618-3.236*\j},0);
}}

\foreach \i in {0,1} {
\foreach \j in {0,1} {
\coordinate (d\i_\j) at ({1.618-3.236*\j},0,{.618-1.236*\i});
}}

\draw[thin,black!50]
(a0_1_0) -- (c0_1)
(d1_0) -- (a0_1_1)
(b0_1) -- (b1_1)
(c1_1) -- (c0_1) -- (a0_1_1) -- (b1_1) -- (a1_1_1)
;

\draw[ultra thick,fill=black!20]
(c1_1) -- (a1_1_0) -- (d0_1) -- (d1_1) -- (a1_1_1) -- cycle
;

\draw[fill=yellow!50,fill opacity=.4]
(a1_0_0) -- (b0_0) -- (b1_0) -- (a1_1_0) -- (d0_1) -- cycle
;

\draw[fill=yellow!40,fill opacity=.4]
(a1_0_0) -- (d0_1) -- (d1_1) -- (a1_0_1) -- (c1_0) -- cycle
;

\draw (-4.236,4.236,4.236) coordinate (e) -- (2,-2,-2);

\shade[ball color=yellow!80, white] (e) circle (.1cm); 

\draw[fill=yellow!30,fill opacity=.4]
(a1_0_0) -- (c1_0) -- (c0_0) -- (a0_0_0) -- (b0_0) -- cycle
;

\draw[fill=blue!40,fill opacity=.4]
(c0_0) -- (a0_0_0) -- (d0_0) -- (d1_0) -- (a0_0_1) -- cycle
;

\draw[ultra thick,fill=black,fill opacity=.4]
(d0_0) -- (a0_0_0) -- (b0_0) -- (b1_0) -- (a0_1_0) -- cycle
;

\draw[ultra thick,fill=black,fill opacity=.4]
(a0_0_1) -- (c0_0) -- (c1_0) -- (a1_0_1) -- (b0_1) -- cycle
;

\draw[ultra thick]
(c0_0) -- (a0_0_0)
(b1_0) -- (a1_1_0)
(a1_0_1) -- (d1_1)
;

\end{scope}

\end{tikzpicture}

\caption{Examples of possible completion conditions to a lamppost polytope for an octagon and for a dodecahedron.  Visibility is indicated by shading.}\label{fig:lamppost}
\end{figure}
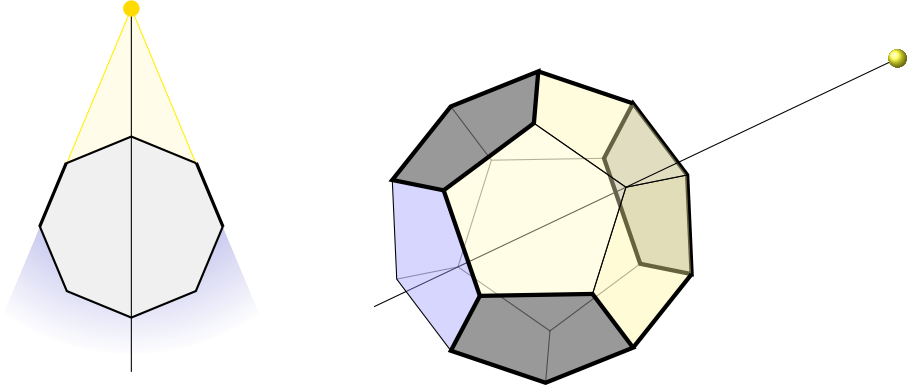

\vbox{%
\begin{lem}\label{lem:lamppost}
Given an abstract lamppost polytope $\mc{L} = \lamp(\chi,f_0,f_1)$ with base $\mc{B}$, 
the completion condition for a realization $B$ of $\mc{B}$ to $\mc{L}$ from its base is that there be a projective copy $\tilde B$ of $B$ and a point $p$ on a line $\ell$ such that $\vis(p,\tilde B,f) = \chi(f)$ for all $f\in\mc{B}$ and the line $\ell$ passes through $\face(\tilde B,f_0)^\circ$ and $\face(\tilde B,f_1)^\circ$.
\end{lem}%
}\nointerlineskip

\begin{proof}
Consider the completion condition of Lemma \ref{lem:projtrans} applied to $\mc{L}$ for a realization $B$ of $\mc{B}$ and an edge $A = \conv\{a_0,a_1\}$. 
Part \ref{condTP} is equivalent to $\vis(p,B,f) = \chi(f)$ for some point $p \in \ell = \projcl(A)$ by Lemma \ref{lem:tent}. 
Part \ref{condTB0} says that the faces $f$ of the base $B$ with a strictly supporting half-space $H_f$ containing $A$ in its interior are precisely the faces that do not contain $f_0$ or $f_1$. 
Part \ref{condTB1} says that no strictly supporting half-space of a vertex $a_i$ contains $B$ in its interior. 

Part \ref{condTB1} is equivalent to $A \subset B$.  With this, part \ref{condTB0} says that $A$ intersects the faces $[f_0,\top]$ and $[f_1,\top]$ of $B$ and no other faces.  Equivalently, $a_i \in \face(B,f_i)^\circ$ possibly reindexing the $a_i$.
Thus, the condition of Lemma \ref{lem:projtrans} applied to the transmitter $\mc{L}$ is exactly the condition of the theorem.
\end{proof}

\subsection{Anchor Polytopes} 

The purpose of the anchor polytope $\anch(\alpha)$ is to fix a point on an edge of the unit square. 
By truncating a vertex of the unit square, we obtain a pentagon with two new vertices, and the anchor polytope fixes the coordinates of one of these new vertices.  
Later the anchor polytope will be used to ``anchor'' a hyperplane by forcing the hyperplane to contain this vertex. 
Before defining anchor polytopes, we state the main lemma for anchor polytopes, which gives completion conditions from a face $f_{\subpentagon} \in \anch(\alpha)$. 

\begin{lem}\label{lem:anchor}
For any algebraic number $0< \alpha < 1$, 
the completion condition for a pentagon $F$ to the anchor polytope $\anch(\alpha)$ from the face $f_{\subpentagon}$ is that $F$ be a projective copy of the unit square truncated at the vertex $(1,1)$ with a vertex at $v_\alpha = (1,\alpha)$. 
\end{lem}

The construction of the anchor polytope will make use of another combinatorial polytope $\mc{R}_\alpha$ 
to encode the value $\alpha$ in a computational frame. 
A computational frame is a $2k$-gon $G$ that satisfies the following. 
For each opposite pair of edges $E_i$, $E_{i'}$ of $G$ let $p_i = \projcl(E_i) \wedge \projcl(E_{i'}) \in \mb{P}^2$. 
A $2k$-gon $G$ is a \df{computational frame} when the points $p_1, \dots, p_k$ are on a line $\Delta \subset \mb{P}^2$, see Figure~\ref{fig:anchorgons} Right. 
We say that a computational frame represents the values $\alpha_1, \dots, \alpha_k \in \mb{P}^1$ when there is a projective transformation $\pi : \Delta  
\to \mb{P}^1$ such that $\pi p_i = \alpha_i$. 
Note that, as along as the images of three points are fixed, this determines the other values represented by the computational frame. 
We will often choose three such points to represent the values 
$\pi p_0 = 0$, $\pi p_1 = 1$, and $\pi p_\infty = \infty$.  
In this case, the value represented by any other point $p_i$ is given by the cross ratio, $\pi p_i = (p_i, p_{1}| p_{0}, p_{\infty})$. 
The following lemma is closely related to \cite[Theorem 8.1.1]{richter1996realization}.

\begin{lem}\label{lem:const} For any positive algebraic number $\alpha$, 
the completion condition for $G$ from $g_\alpha$ to 
$\mc{R}_\alpha$ is that $G$ be a computational frame representing $0,\alpha, 1, \infty$. 
\end{lem}

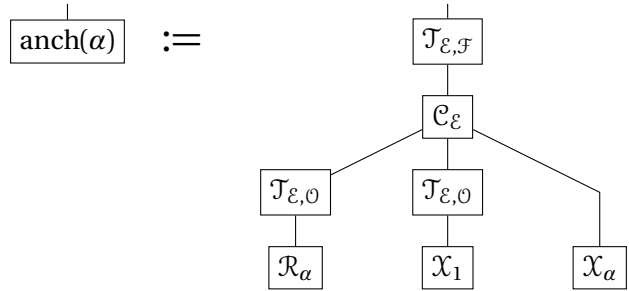
\begin{figure}[h!]
\begin{center}

\begin{tikzpicture}
 \node at (0,1) [draw] (tec) {$\mc{T_{E,F}}$};
 \node at (0,0) [draw] (cn) {$\mc{C_E}$};
 \node at (2,-2) [draw] (xa) {$\mc{X}_\alpha$};
 \node at (0,-1) [draw] (tnx) {$\mc{T_{E,O}}$};
 \node at (0,-2) [draw] (x1) {$\mc{X}_1$};
 \node at (-2,-1) [draw] (tno) {$\mc{T_{E,O}}$};
 \node at (-2,-2) [draw] (r) {$\mc{R}_\alpha$};

\draw   (0,1.5) -- (tec) -- (cn)
        (cn) -- (2,-1) -- (xa)
        (cn) -- (tnx) -- (x1)
        (cn) -- (tno) -- (r);
        
        \node at (-5,1) [draw] (a) {$\anch(\alpha)$};
        \node at (-3.5,1) {\scalebox{1.5}{$:=$}};
        
\draw (-5,1.5) -- (a);
\end{tikzpicture}

\caption{Gluing diagram for an anchor polytope of $\alpha$.}\label{glue:anchor}
\end{center}
\end{figure}

\begin{figure}[ht]
\centering

\begin{tikzpicture}	

\matrix[column sep=1.5cm]
{
\filldraw[fill=green!20] (0,0) -- (1.8462,0.4615) -- (1.6364, 1.0909) -- (1,1.75) -- (0.6154,1.8462) -- cycle;
\fill[fill=black] (1.6364, 1.0909) circle (.06) node[right,shift={(-1pt,-2pt)}] {\footnotesize $v_\alpha$} ;

  \node at (0,3.6667) (ld) {
{$\Delta$}};
  \node at (-1,2.25) (lx') {
{$\ell_{0'}$}};
  \node at (-1,-.25) (lx) {
{$\ell_0$}};
  \node at (2.3333,-1) (ly') {
{$\ell_{\infty'}$}};
  \node at (-.3333,-1) (ly) {
{$\ell_\infty$}};
  \draw (ld) -- (4.75,.5)
  (lx) -- (4,1)
  (lx') -- (4,1)
  (ly) -- (1,3)
  (ly') -- (1,3)
  (0,0) -- (2.0465,2.3023)  
  (0,0) -- (2.75,1.8333) 
;

\node at (.15,-.15) {\footnotesize $v_0$};
\node at (1.7,1.7) {\footnotesize $v_1$};
\node at (4.14,1.15) {
{ $p_0$}};
\node at (2.9,2) {
{ $p_\alpha$}};
\node at (2.2,2.46) {
{ $p_1$}};
\node at (1.2,3.15) {
{ $p_\infty$}};

&

\filldraw[fill=blue!20] (0,0) -- (0.8919,0.2230) -- (1.5,.75) -- (1.6284,1.1148) -- (1.4545,1.6364) -- (1.1141,1.7215) -- (.8,1.7) -- (0.5219,1.5656) -- cycle;

  \node at (0,3.6667) (ld) {
{$\Delta$}};

  \draw (ld) -- (4.75,.5)
  (0,0) -- (4,1)
  (1.1141,1.7215) -- (4,1)
  (0,0) -- (1,3)
  (1.6284,1.1148) -- (1,3)
  (2.75,1.8333) -- (.8,1.7)
  (2.0465,2.3023) -- (0.5219,1.5656)
  (2.0465,2.3023) -- (1.5,.75)
  (2.75,1.8333) -- (1.5,.75)
;

\node at (4.14,1.15) {
{ $p_0$}};
\node at (2.9,2) {
{ $p_\alpha$}};
\node at (2.2,2.46) {
{ $p_1$}};
\node at (1.2,3.15) {
{ $p_\infty$}};

\\
};

\end{tikzpicture}

\caption{
\textsc{Left:} A realization of the pentagon $\mc{F}$ in the construction of $\anch(\alpha)$ with collinearities shown;  
a truncated unit square with vertex at $v_\alpha=(1,\alpha)$.
\textsc{Right:} A realization of the octagon $\mc{O}$ in the construction of $\anch(\alpha)$;
a computational frame representing the values $0,\alpha,1,\infty$.}\label{fig:anchorgons}
\end{figure}
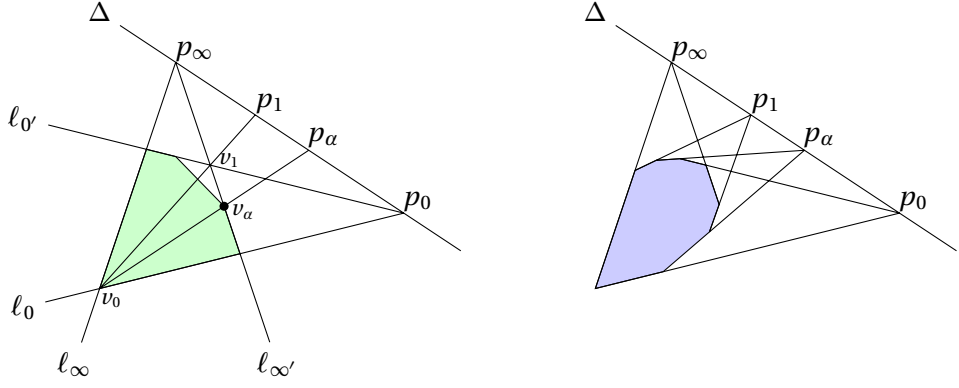

We will now construct the anchor polytope $\anch(\alpha)$ using $\mc{R}_\alpha$ and prove Lemma~\ref{lem:anchor}.  Later we will construct $\mc{R}_\alpha$ and prove Lemma~\ref{lem:const}. 
The \df{combinatorial anchor polytope} $\anch(\alpha)$ consists of 
combinatorial polytopes glued together as in Figure~\ref{glue:anchor}. 
Let $\mc{E}$, $\mc{F}$, and $\mc{O}$ respectively be a combinatorial enneagon (9-gon), pentagon, and octagon, and label their edges consecutively as follows
\[
\begin{array}{r@{\ \ \text{ by }\ \ }l@{\ }l@{\ }l@{\ }l@{\ }l@{\ }l}
\mc{E} & \infty,\ 0, & \alpha,\ 1, & \infty', & h, & 0', & \alpha',\ 1', \\
\mc{F} & \infty,\ 0, &  & \infty', & h, & 0', &  \\
\mc{O} & \infty,\ 0, & \alpha,\ 1, & \infty', & & 0', & \alpha',\ 1'. \\
\end{array} 
\]
We may consider $\mc{E}$ as a copy of $\mc{F}$ with the vertices $0 \wedge \infty'$ and $0' \wedge \infty$ whittled by the edges $\alpha$, $1$ and $\alpha'$, $1'$ respectively. 
This defines a forgetful transmitter $\mc{T}_{\mc{E},\mc{F}}$ between $\mc{E}$ and $\mc{F}$ that ``forgets'' these extra edges. 
Similarly, 
let $\mc{T_{E,O}}$ be the forgetful transmitter that ``forgets'' the edge $h$. 
Let $\mc{C}_\mc{E} = \conn(4,\mc{E})$. 
One pyramidal facet of $\mc{C_E}$ is glued to the forgetful transmitter $\mc{T_{E,F}}$, and the remaining base face of $\mc{T_{E,F}}$ is then the specified face $f_{\subpentagon}$ of Lemma~\ref{lem:anchor}.
A pyramidal facet of $\mc{C_E}$ is glued to the forgetful transmitter $\mc{T_{E,O}}$, 
which is then glued to $\mc{R}_\alpha$ so that the octagon $\mc{O}$ is identified with the computational frame of Lemma~\ref{lem:const} where the pairs of edges $\{0,0'\}$, $\{\alpha,\alpha'\}$, $\{1,1'\}$, and $\{\infty,\infty'\}$ represent the values $0$, $\alpha$, $1$, and $\infty$ respectively; Figure~\ref{fig:anchorgons} Right. 
Another pyramidal facet of $\mc{C_E}$ is glued to the forgetful transmitter $\mc{T_{E,O}}$, 
which is then glued to the lamppost polytope $\mc{X}_1 = \lamp(\chi_1,v_0,v_1)$ 
where $v_0 = \infty \wedge 0$ and $v_1 = \infty' \wedge 0'$, and 
$\chi_1(1) = \chi_1({1'}) = 0$, and $\chi_1(f) = -$ for $f \in \{\infty,0,\alpha\}$, and $\chi_1(f) = +$ otherwise. 
Note that by Lemma~\ref{lem:lamppost}, the completion condition to $\mc{X}_1$ imposes a collinearity between $v_0$, $v_1$, and $p_1$; Figure~\ref{fig:anchorgons} Left. 
Finally the last pyramidal facet of $\mc{C_E}$ is glued to the lamppost polytope 
$\mc{X}_\alpha = \lamp(\chi_\alpha,v_0,v_\alpha)$ where $v_\alpha = \infty' \wedge h$, and 
$\chi_\alpha(\alpha) = \chi_\alpha({\alpha'}) = 0$, and $\chi_\alpha(f) = -$ for $f \in \{1',\infty,0\}$, and $\chi_\alpha(f) = +$ otherwise. 
Note that by Lemma~\ref{lem:lamppost}, the completion condition to $\mc{X}_\alpha$ imposes a collinearity between $v_0$, $v_\alpha$, and $p_\alpha$; see Figure~\ref{fig:anchorgons} Left.

\begin{proof}[Proof of Lemma~\ref{lem:anchor}]
In the construction of $\anch(\alpha)$ we always glue along pyramids so by Lemmas~\ref{lem:glue} and~\ref{lem:unglue} a pentagon $F$ can be completed from $f_{\subpentagon}$ if and only if we can find respective realizations 
$E$ and $O$ of $\mc{E}$ and $\mc{O}$ such that these polygons, together with $F$ realizing $\mc{F}$, satisfy the completion conditions of the pieces. 
Let $\ell_i$ be the line supporting the edge labeled $i$, and let $p_i = \ell_{i} \wedge \ell_{i'}$.

To see the given condition is sufficient let $F$ be the unit square with $(1,1)$ truncated by a line $\ell_h$ passing though $(1,\alpha)$.  Let $E$ be a enneagon that results from $F$ when the vertices $(1,0)$ and $(0,1)$ are each truncated by a pair of lines with slope $1$ and slope $\alpha$.  Let $O$ be the octagon that results when the linear constraint imposed on $E$ by the supporting line $\ell_h$ is dropped.  With this $p_0$ is the point where all horizontal lines meet, $p_\alpha$ is on all lines with slope $\alpha$, $p_1$ is on all lines with slope 1, and $p_\infty$ is on all vertical lines.  As such $O$ is a computational frame representing $0,\alpha,1,\infty$, the points $v_0 = (0,0)$, $v_1 = (1,1)$, and $p_1$ are collinear, and the points $v_0 = (0,0)$, $v_\alpha = (1,\alpha)$, and $p_\alpha$ are collinear.  Thus, all the completion conditions of the pieces are satisfied, which means the given condition is sufficient.

To see that the given condition is necessary, suppose we have a realization of $\anch(\alpha)$. 
We start by choosing a projective coordinate system so $F$ is a truncated unit square with $(v_\alpha)_1=1$.  For this we let the region bounded by $\ell_\infty$, $\ell_0$, $\ell_{\infty'}$, $\ell_{0'}$ be the unit square with 
\[ \ell_\infty = \{(x,y):x=0\},\ \ell_0 = \{(x,y):y=0\},\ \ell_{\infty'} = \{(x,y):x=1\},\  \ell_{0'} = \{(x,y):y=1\} .\]  
The polytope $\mc{R}_\alpha$ forces the points $p_0$, $p_\alpha$, $p_1$, and $p_\infty$ to be on a line $\Delta$, which is the horizon in this coordinate system, and $\mc{R}_\alpha$ forces $(p_\alpha,p_1|p_0,p_\infty)=  \alpha$. 
The lamppost polytope $\mc{X}_1$ forces $v_0 = (0,0)$, $v_1 = (1,1)$, $p_{1}$ to be collinear, and $\mc{X}_\alpha$ forces $v_0 = (0,0)$, $v_\alpha$, $p_\alpha$ to be collinear. 
As a consequence of these collinearities, the projection from $\ell_{\infty'}$ to $\Delta$ through $v_0$ sends $(1,0)$, $v_\alpha$, $(1,1)$, $p_\infty$ to $p_0$, $p_\alpha$, $p_1$, $p_\infty$ respectively, which implies
\[ (v_\alpha)_2 =(v_\alpha,(1,1)|(1,0),p_\infty) =(p_\alpha,p_1|p_0,p_\infty)=  \alpha.  \]
Thus, the condition of the theorem is necessary. 
\end{proof}

To construct $\mc{R}_\alpha$ we use certain combinatorial 4-polytopes that encode the basic arithmetic relations (addition and multiplication) in a computational frame through its completion condition. 
The construction of $\mc{R}_\alpha$ will be 
similar to that of the polytope in \cite[Theorem~8.1.1]{richter1996realization}, where the following \df{basic arithmetic polytopes} were introduced in order to encode a system of polynomials in a combinatorial 4-polytope.  
For the construction of the four basic arithmetic polytopes and the proof of Lemma~\ref{lem:arith} see  
\cite[Lemmas 6.2.1, 7.1.1, 7.2.1, and 7.2.2]{richter1996realization}.

\begin{lem}\label{lem:arith}
There exist combinatorial 4-polytopes $\mc{R}_{2x}$, $\mc{R}_{x+y}$, $\mc{R}_{x^2}$, $\mc{R}_{xy}$ that each have a pyramidal facet with base $g$ such that 
the completion condition of $G$ from $g$ to each of these is that there exist $x,y,z \in \mb{R}$ such that $G$ is a computational frame representing the following sequences values:
\[ \begin{array}{l@{\,}l@{\ \text{ for }\ }l}
 0 < x < 2x  & < \infty & \mc{R}_{2x}, \\
 0 < x < y < x+y  & < \infty & \mc{R}_{x+y}, \\
 0 < 1 < x < x^2 & < \infty & \mc{R}_{x^2}, \\
 0 < 1 < x < y < x y  & < \infty & \mc{R}_{xy}. \\
\end{array}  \]
\end{lem}

We can apply certain affine transformations to $\mb{P}^1$ and make substitutions to change the values represented by a computational frame.  For example, if we apply $\phi(t) = t-x$ to the values represented by the computational frame of $\mc{R}_{2x}$, and make the substitution $x' = \phi(2x)$, we get a computational frame representing $-x' < 0 < x' < \infty$.  In this way, we can use the four basic arithmetic polytopes to represent the respective inequalities and algebraic relations in the following table, together with the value $\infty$.  
Each row below a basic arithmetic polytope in Table \ref{tab:arith} lists a possible relation imposed by that polytope with values given in increasing order from left to right. 

\[
\begin{array}{ccc|cccc|cccc|ccccc}
\multicolumn{3}{c|}{\mc{R}_{2x}} & \multicolumn{4}{c|}{\mc{R}_{x+y}}
 & \multicolumn{4}{c|}{\mc{R}_{x^2}} & \multicolumn{5}{c}{\mc{R}_{xy}} \\
\hline
0 & x & 2x & 0 & x & y & x{+}y & 0 & 1 & x & x^2 & 0 & 1 & x & y & xy \\
-x & 0 & x & x & 0 & x{+}y & y & 0 & x^{-1} & 1 & x & 0 & x & 1 & xy & y \\
2x & x & 0 & x & x{+}y & 0 & y & 0 & x^2 & x & 1 & 0 & x & xy & 1 & y \\
 & & & x{+}y & x & y & 0 & & & & & 0 & xy & x & y & 1 \\
\end{array}
\]

\captionof{table}{Possible algebraic relations and inequalities imposed by each basic arithmetic polytope.
}\label{tab:arith}

\bigskip

We will combine the completion conditions of basic arithmetic polytopes to force a computational frame to represent $0,\alpha,1,\infty$ along with several other values.  We then use a forgetful transmitter to ``forget'' these extra values. 
Let $p(x) = \sum_{t=0}^n c_t x^t$ be the minimal polynomial of $\alpha$, and let $b_1,b_2 \in \mb{N}$ such that $\alpha$ is the only real root of $p$ that satisfies $b_1-1 < b_2\alpha < b_1$. 
Let $\{x_1< \dots < x_m\}$ consist of the following values: 
the integers from 0 to $N = \max\{b_i,|c_t|\}$, the powers of $\alpha$ up to $\alpha^n$, the absolute values of all monomials $|c_t \alpha^t|$, the values of all partial sums $\sum_{s=1}^t c_s \alpha^s$ for $c_t \neq 0$, and $b_2\alpha$. 
We now define collections of pairs of indices $\mathbf{R}_{2x}$, $\mathbf{R}_{x^2}$ and triples of indices $\mathbf{R}_{x+y}$, $\mathbf{R}_{xy}$ corresponding to arithmetic relations that determine the values $x_i$. 
\[ 
\begin{array}{r@{\ \in \ }l@{\ \text{ for }\ }r@{\ }ll}
\{i,j\} & \mathbf{R}_{2x} & (x_i,x_j) & =(1,\ 2), \\
\{i,j,k\} & \mathbf{R}_{x+y} & (x_i,x_j,x_k) &=(1,\ t,\ t+1), & t \in \{2,\dots, N-1\} \\
\{i,j,k\} & \mathbf{R}_{xy} & (x_i,x_j,x_k)&=(b_2,\ \alpha,\ b_2\alpha), & b_2 \neq 1 \\
\{i,j\} & \mathbf{R}_{x^2} & (x_i,x_j)&=(\alpha,\ \alpha^2), \\
\{i,j,k\} & \mathbf{R}_{xy} & (x_i,x_j,x_k)&=(\alpha,\ \alpha^{t},\ \alpha^{t+1}), & t \in \{2,\dots, n-1\} \\
\{i,j,k\} & \mathbf{R}_{xy} & (x_i,x_j,x_k)&=(|c_t|,\ \alpha^{t},\ |c_t\alpha^{t}|), & c_t \neq 0, t<n \\
\{i,j,k\} & \mathbf{R}_{x+y} & 
(x_i,x_j,x_k)&=(|c_{t}\alpha^{t}|,\ \sum_{s=1}^{t-1} c_s\alpha^{s},\ \sum_{s=1}^{t} c_s\alpha^{s}), & c_t \neq 0, t<n \\
\{i,j\} & \mathbf{R}_{2x} & (x_i,x_j)&=(c_{n}\alpha^{n},\ -c_n\alpha^n = \sum_{s=1}^{n-1} c_s\alpha^{s}). \\
\end{array}
\]
Let $\{o,a,u\} \subset \{1,\dots,m\}$ for which $x_{o} = 0$, $x_{a} = \alpha$, and $x_{u} = 1$. Note that since $p$ is the minimal polynomial of $\alpha$, no partial sum vanishes, $\sum_{s=1}^{t} c_s\alpha^{s} \neq 0$, so $o \nin \bigcup \mathbf{R}_{x+y}$.  Also $o \nin \bigcup \mathbf{R}_{2x}$ and $o,u \nin \bigcup \mathbf{R}_{xy} \cup \bigcup \mathbf{R}_{x^2}$.

In the construction of $\mc{R}_\alpha$, we will glue basic arithmetic polytopes along computational frames according to the above collections.  
Every computational frame will contain a pair of edges that represent $\infty$, and these pairs will always be identified when gluing.  The order of the other values represented then determines the gluing, and also determines the relation imposed by each basic arithmetic polytope. 
For example, if $\mc{R}_{x+y}$ is glued along a computational frame representing $x_i < x_j < x_o = 0 < x_k < \infty$, then the arithmetic relation imposed on these values is $x_j = x_i+x_k$, since 0 is the third value; see row 3 of column $\mc{R}_{x+y}$ in Table \ref{tab:arith}.

Let $\mc{G}$ be a combinatorial polygon with $2m+2$ edges indexed consecutively $1,\dots,m,\infty,1',\dots,m',\infty'$, and $\mc{C}_\mc{G} = \conn(\mc{G})$. 
Let $\mc{T}_{i,j}$ be a forgetful transmitter between $\mc{G}$ and an octagon with edges indexed by  
$\{o,o',\infty,\infty',i,i',j,j'\}$, and 
let $\mc{T}_{i,j,k}$ and $\mc{T}_{i,j,k,l}$ be similarly defined forgetful transmitters respectively between $\mc{G}$ and a decagon and a dodecagon.
Finally, let $\mc{R}_\alpha$ be given by the gluing diagram in Figure~\ref{glue:R}.

\begin{figure}[h]
\centering

\begin{tikzpicture}
 \node at (0,1) [draw] (t0) {$\mc{T}_{u,a}$};
 \node at (0,0) [draw] (cg) {$\mc{C}_{\mc{G}}$};
 \node at (2,-1) [draw] (t1) {$\mc{T}_{i,j}$};
 \node at (2,-2) [draw] (p1) {$\mc{R}_{2x}$};
 \node at (4,-1) [draw] (t2) {$\mc{T}_{i,j,k}$};
 \node at (4,-2) [draw] (p2) {$\mc{R}_{x+y}$};
 \node at (6,-1) [draw] (t3) {$\mc{T}_{u,i,j}$};
 \node at (6,-2) [draw] (p3) {$\mc{R}_{x^2}$};
 \node at (8,-1) [draw] (t4) {$\mc{T}_{u,i,j,k}$};
 \node at (8,-2) [draw] (p4) {$\mc{R}_{xy}$};
 \node at (2,-3.5) (b1) {
$\begin{array}{c} \vdots \\ \mathbf{R}_{2x} \\ \vdots \end{array}$ };
 \node at (4,-3.5) (b2) {
$\begin{array}{c} \vdots \\ \mathbf{R}_{x+y} \\ \vdots \end{array}$ };
 \node at (6,-3.5) (b3) {
$\begin{array}{c} \vdots \\ \mathbf{R}_{x^2} \\ \vdots \end{array}$ };
 \node at (8,-3.5) (b4) {
$\begin{array}{c} \vdots \\ \mathbf{R}_{xy} \\ \vdots \end{array}$ };
 \node [left=.3 of t0] (eq) {\scalebox{1.5}{$:=$}};
 \node [left=.3 of eq, draw] (r) {$\mc{R}_{\alpha}$};

 \path (1,0) coordinate (a1)
       (3,0) coordinate (a2) 
       (5,0) coordinate (a3) 
       (7,0) coordinate (a4);

\draw (cg) -- (t0) -- +(0,.5)
      (r)  -- +(0,.5)
      (cg) -- (a4);
\foreach \i in {1,...,4}
 {
 \draw (a\i) -- +(0,-4.5)
       +(0,-3) -- +(.3,-3)
       +(0,-3.5) -- +(.3,-3.5)
       +(0,-4) -- +(.3,-4)
       +(0,-1) -- (t\i) -- (p\i);
 }

\end{tikzpicture}
\caption{Gluing diagram for $\mc{R}_\alpha$.}\label{glue:R}
\end{figure}
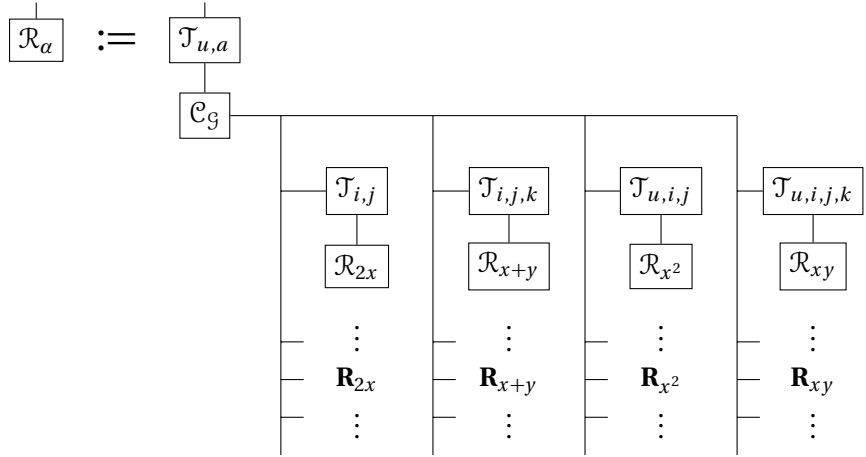

\begin{proof}[Proof of Lemma~\ref{lem:const}]
In the construction of $\mc{R}_\alpha$ we always glue along pyramids so by Lemmas \ref{lem:glue} and \ref{lem:unglue} a polygon $G_\alpha$ can be completed from $g_\alpha$ to $\mc{R}_\alpha$ if and only if each of the pieces glued together in the construction can be realized such that each pair of facets glued along is a projectively equivalent pair and the face $g_\alpha$ is a projective copy of $G_\alpha$. By Lemmas \ref{lem:ftrans} and \ref{lem:cons} and by Lemma~\ref{lem:arith}, this is equivalent to the existence of a polygon $G$ with edge supporting lines $l_1,\dots,l_m,l_\infty,l_{1'},\dots,l_{m'},l_{\infty'}$ such that the polygon bounded by $l_{o},l_{a},l_{u},l_{\infty},l_{o'},l_{a'},l_{u'},l_{\infty'}$ is $G_\alpha$, and for each set of indices in 
$\mathbf{R}_\phi$ for $\phi = 2x$, $x{+}y$, $x^2$, and $xy$, 
the polygon defined by the corresponding lines is a computational frame representing values that satisfy the corresponding algebraic relations and inequalities. 
Let $p_i = l_{i} \wedge l_{i'}$ and $\Delta = p_{o} \vee p_\infty$. 
Every index $i \neq o, \infty$ appears in at least one of the given relations, and since the corresponding polygon is a computational frame with $o,o',\infty,\infty'$ among its edges, this puts $p_i \in \Delta$.  
Therefore, $G_\alpha$ can be completed to $\mc{R}_\alpha$ if and only if there is a computational frame $G \subset G_\alpha$ representing values $x_1,\dots,x_m,x_\infty$ where $x_{o} = 0$, $x_{u} = 1$, $x_{\infty} = \infty$ such that all the relations of the $\mathbf{R}_\phi$ are satisfied, and $G_\alpha$ is the polygon defined by the lines corresponding to $x_o,x_\alpha,x_u,x_\infty$.

In any realization of $\mc{R}_\alpha$, the value $x_{a}$ must satisfy $p(x_{a}) = 0$ and $b_1-1 < b_2 x_{a} < b_1$.  These were chosen to leave only one possibility $x_{a} = \alpha$.  Hence, $G_\alpha$ must be a computational frame representing $0,\alpha,1,\infty$, so the condition of the theorem is necessary. 

Suppose $G_\alpha$ is a computational frame representing $0,\alpha,1,\infty$.  Then, $G_\alpha$ can be truncated to produce a computational frame representing values $x_1,\dots,x_m,x_{\infty}$ that satisfy all of the relations of the $\mathbf{R}_\phi$.  Hence, $G_\alpha$ can be completed to a realization of $\mc{R}_\alpha$, so the condition of the theorem is sufficient.  
\end{proof}

\subsection{Stamps} 

We now construct the stamp $\mc{S}_P$ of a polytope $P \subset \ralg^d$ from Theorem~\ref{the:paf}. 
Our goal is to fix a face $F_P = \face(S,f_P)$ up to projectivity for all realizations $S$ of $\mc{S}_P$ such that $F_P \projeq P$.  We do so by combining a sequence of completion conditions such that the conjunction of all these conditions is equivalent to the condition $F_P \projeq P$.  That is, the combined conditions determine $P$ up to projectivity. 
For every facet $f$ of $P$, we give a sequence of conditions that determine the supporting hyperplane $h_f$ of the facet $F = \face(P,f)$.  For this, we give the coordinates of a set of $d$ points in $h_f$ in general position.  
We will assume $P$ is positioned ``nicely'' in the sense 
that 
each of these points $p$ is the intersection of $h_f$ with an edge $e$ of the unit cube.  The edge $e$ determines all coordinates of $p$ except one, and to fix this last coordinate as part of the completion conditions of $F_P$, we use an anchor polytope.  An anchor polytope fixes a coordinate of a vertex of a pentagon.  The following lemmas ensures we can always find an appropriate pentagon. 

\vbox{
\begin{lem}
\label{lem:pent}
If ${H}$ is a closed half-space such that the following holds
\begin{itemize}
\item
$H$ does not contain the unit $d$-cube~ $\cube$, 
\item
$H$ intersects every facet of~ $\cube$, 
\item
no vertices of~ $\cube$ are on the boundary $h = \partial H$, 
\end{itemize}
then the polytope $Q_H := \cube\cap H$ has at least $d$ distinct vertices on the hyperplane $h$ that are each a vertex of some pentagonal face of $Q_H$. 
\end{lem}
}

Let $\mathbf{F}$ be the facets of $P$, and for each $f \in \mathbf{F}$, and 
let $H_f$ be the half-space that contains $P$ and strictly supports $F = \face(P,f)$. 

\begin{lem}\label{lem:nicepose}
There is a rational affine transformation such that for all facets $f \in \mathbf{F}$, $H_f$ satisfies the hypothesis of Lemma~\ref{lem:pent}. 
\end{lem}

Assume that $P$ has full dimension, and by Lemma~\ref{lem:nicepose}, that the strictly supporting half-space of each facet satisfies the hypothesis of Lemma~\ref{lem:pent}. 
Let $Q = \cube \cap P$ and $\mc{C}_Q = \conn(\labl(Q))$ and $\mc{T}_{Q,P}$, $\mc{T}_{Q,\subcube}$ be the combinatorial types of $\trmr(Q,P)$ and $\trmr(Q,\cube)$ respectively. 
Let $\bfpentagon_f$ be the triples $(\alpha,p,\pentagon)$ where $(p,\pentagon)$ are the (vertex, pentagonal face) pairs of $Q_f = \cube \cap H_f$ implied by Lemma~\ref{lem:pent}, and $0<\alpha<1$ is the only coordinate of $p$ that is not 0 or 1.  
Let $\mc{T}_{Q,Q_f}$ be the combinatorial type of $\trmr(Q,Q_f)$.
For each triple $(\alpha,p,\pentagon) \in \bfpentagon_f$, let $\mc{A}_{Q_f,(\alpha,p,\pentagon)}$ be the combinatorial polytope obtained by gluing $\pyr^{d-2}\anch(\alpha)$ to $\adapt(\labl(Q_f),\pentagon)$ along the face $\pentagon$, so that the unit square of Lemma~\ref{lem:anchor} is identified with a face of $\cube$, and the coordinates of the vertex $p$ are fixed by the anchor polytope. 
Finally, let $\mc{S}_P$ be the combinatorial polytope obtained by gluing these pieces together according to Figure~\ref{fig:stampdia}.


\begin{figure}[h]
\begin{center}

\begin{tikzpicture}
 \node at (0,0) [draw] (tfg) {$\mc{T}_{Q,P}$};
 \node at (0,-1) [draw] (cg) {$\mc{C}_{Q}$};
 \node at (0,-2) [draw] (tc) {$\mc{T}_{Q,\subcube}$};
 \node at (0,-3) [draw] (u) {$\mc{S}_\subcube$};
 \node at (2,-2) [draw] (tgh) {$\mc{T}_{Q, Q_f}$};
 \node at (2,-3) [draw] (ch) {$\mc{C}_{Q_f}$};
 \node at (4,-4) [draw] (the) {$\mc{A}_{Q_f,(\alpha,p,\subpentagon)}$};
 \node at (4.5,-2) {$ \cdots\ \mathbf{F}\ \cdots$};
 \node at (6.5,-4) {$\cdots\ \bfpentagon_f\ \cdots$};

 \draw	(tfg) -- (cg) -- +(6,0)
  (cg) -- (tc) -- (u)
	(3.5,-1) -- +(0,-.5)
	(4.5,-1) -- +(0,-.5)
	(5.5,-1) -- +(0,-.5)
	(2,-1) -- (tgh) -- (ch) -- +(6,0)
	(5.5,-3) -- +(0,-.5)
	(6.5,-3) -- +(0,-.5)
	(7.5,-3) -- +(0,-.5)
	(4,-3) -- (the);

	\node at (-2.2,0) [draw] {$\mc{S}_P$};
	\node at (-1.2,0) {\scalebox{1.5}{$:=$}};
\end{tikzpicture}
\caption{Gluing diagram for a stamp of $P$.}\label{fig:stampdia}
\end{center}
\end{figure}
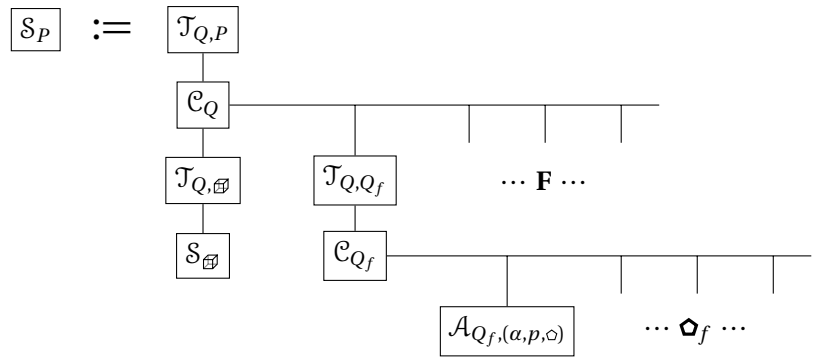

\begin{proof}[Proof of Lemma~\ref{lem:nicepose}]
First choose a coordinate system that is generic with respect to $P$, and let $v_i$ and $w_i$ be the vertices of $P$ with the respectively least and greatest $i^\text{th}$ coordinate.  Note that some vertices may be repeated.  For each such vertex $u \in \{v_1,\dots,w_d\}$, let $\bar{u} \in P^\circ$ such that $\bar{u}$ is rational, $\bar{u}$ is sufficiently close to $u$ that the $i^\text{th}$ coordinate all vertices of $P$ except $v_i$ and $w_i$ are in the interval $I_i = [(\bar{v}_i)_i,(\bar{w}_i)_i]$, and no vertex of $C = \prod_{i=1}^d I_i$ is in the supporting hyperplane of any facet of $P$. 
We now get the desired affine transformation defined by sending $C$ to $\cube$. 
\end{proof}

\begin{proof}[Proof of Lemma~\ref{lem:pent}]
We index the nonempty faces of the unit cube $\cube$ by ${\{0,1,*\}^d}$ where $0,1 \leq *$ and $0,1$ are incomparable.  For this, 0 and 1 indicate coordinates that are constant on a face and $*$ indicates the coordinates of a face that can vary.  
Let $\mathbf{C}_H$ be the set of maximal faces of $\cube$ that do not intersect $H$.  
We may assume without loss of generality that 
the outward normal vector of $H$ is in the positive orthant, which implies  
${\mathbf{C}_H \subset \{1,* \}^d}$.  
Otherwise, we can change coordinates so that this assumption holds. 

For a string $f \in \{0,1,x,y\}^d $, let $\lambda f : \{0,1,*\}^2 \to \{0,1,*\}^d$ such that $\lambda f(\tilde x, \tilde y)$ is the face of $\cube$ obtained by substituting $\tilde x,\tilde y$ for $x,y$ respectively in the string $f$. 
Observe that each vertex $p \in Q_H$ on the hyperplane $h = \partial H$ can be identified by the edge $e$ of $\cube$ that contains $p$, and each pentagonal face $\pentagon \subset Q_H$ can be identified by the square $s$ of $\cube$ that contains $\pentagon$. 
We label each such pair 
$(p,\pentagon)$ where $p \in \pentagon$ by $f \in \{0,1,x,y\}^d $ such that  
$f(1,*) = e$ and $f(*,*) = s$ are respectively the edge and square of $\cube$ that contain $p$ and $\pentagon$. 

Now 
let ${I:=\bigcap_{c\in \mathbf{C}_H} \{i:c_i=*\} }$ be the set of coordinate where all faces of $\mathbf{C}_H$ have the value $*$.  We claim that for each ${c\in \mathbf{C}_H}$ and ${J\subset I}$ and ${i_x,i_y \in \{i:c_i=1\}}$ 
with ${i_x\neq i_y}$ there is a pair $(p,\pentagon)$ labeled by 
$f = f({c,J,i_x,i_y})$ where 
\[ f_i = \left\{ \begin{array}{cl}
 x & i=i_x \\
 y & i=i_y \\
 0 & c_i=*,\ i\notin J \\
 1 & \text{else}.
\end{array} \right.  \]

We must show that exactly one of the vertices of the corresponding square $\lambda f(*,*)$ of $\cube$ is not in $H$.  Specifically, we will show that the vertex $\lambda f(1,1)$ is not in $H$ and $\face(\cube,v) \in H$ for 
$v = \lambda f(0,1)$, $\lambda f(1,0)$, $\lambda f(0,0)$. 
For every coordinate $1\leq i \leq d$, $c_i \in \{1,*\}$ and if $c_i = 1$ then $\lambda f(1,1)_i = 1$, so $\lambda f(1,1)$ is a vertex of $c$, 
which implies $\face(\cube, \lambda f(1,1)) \nin H$. 
For any other vertex $v$ of $\lambda f(*,*)$, either $v_{i_x} = 0$ or $v_{i_y} = 0$, but $c_{i_x} = c_{i_y} = 1$, so $v$ is not a vertex of $c$.  For any other $c' \in \mathbf{C}_H$, $c'\neq c$, there is some coordinate $1 \leq i \leq d$ such that $c'_i = 1$ but $c_i = *$, since the faces in $\mathbf{C}_H$ are incomparable.  Since $c'_i \neq *$ We have $i \nin I$, so $i \nin J$, which implies $v_i = 0$, so $v$ not a vertex of $c'$.  We have now that $v$ is not a vertex of any face in $\mathbf{C}$, so $\face(\cube,v) \in H$.

We now find $d$ distinct points where $h$ intersects a distinct edge of $\cube$ of the form $\lambda f(1,*)$.  Observe that for any $c \in \mathbf{C}_H$, there are always at least two coordinates $i_x, i_y$ such that $c_{i_x} = c_{i_y} = 1$, since $c$ would otherwise either be a facet of $\cube$ that is disjoint from $H$ or be the cube itself. 
For $1 \leq k \leq d$, select ${f^k=f(c^k,J^k,i_x^k,i_y^k)}$ in the following way. 
For ${k \in I}$, let ${c^k}$ be any element of $\mathbf{C}_H$ and ${i_x^k,i_y^k}$ be any distinct indices in $\{i:c^k_i=1\}$ and ${J^k=\{k\}}$.  
For ${k \notin I}$, let $c^k \in \mathbf{C}_H$ such that ${c^k_k =1}$ and ${i_y^k=k}$ and ${i_x^k} \in \{i \neq i_y^k:c^k_i=1\}$ and ${J^k=\emptyset}$.

For ${k \in I}$ and any $j \neq k$, we have $\lambda f^k(1,*)_k = 1$ but $\lambda f^{j}(1,*)_k = 0$, so $\lambda f^k(1,*)$ and $\lambda f^{j}(1,*)$ are distinct edges.  Alternatively for $k,j \nin I$ with $j \neq k$, we have $\lambda f^k(1,*)_k = *$ but $\lambda f^{j}(1,*)_k \neq *$ since $\lambda f^j(1,*)_j = *$, so again $\lambda f^k(1,*)$ and $\lambda f^{j}(1,*)$ are distinct edges.
\end{proof}

\begin{proof}[Proof of Theorem \ref{the:paf}]
First we see that $\mc{S}_P$ is always realizable.  As long as we can realize each piece such that every pair of facets we glue along is projectively equivalent, by Lemma~\ref{lem:glue} we can actually glue them together to get a realization of $\mc{S}_P$.  All of these facets are pyramids, so a pair is projectively equivalent if and only if a their respective bases, which are ridges of $\mc{S}_P$, are projectively equivalent. 
In defining $\mc{S}_P$ we start with polytopes $P$ and $\cube$, and define the combinatorics of the various projective transmitters to be realizable with projective copies of these, so the projective transmitters $\mc{T}_{Q,P}$, $\mc{T}_{Q,\subcube}$, $\mc{T}_{\subcube,Q_f}$ can be realized with the specified bases projectively equivalent to $P$, $Q$, $\cube$, $Q_f$.  The stamp of the cube can be realized with the specified ridge projectively equivalent to $\cube$, and for each pentagon in $\bfpentagon_{\bf H}$,  the anchor polytope $\anch(\alpha)$ can be realized so its specified ridge is projectively equivalent to the corresponding pentagonal face of $Q_f$. 

For the other direction we will show that in every realization $S$ the specified facet $F_P = \face(S,f_P)$ is projectively equivalent to $P$.  
Note that, since each facet we glue along to construct $\mc{S}_P$ is a pyramid, these facets are necessarily flat, which by Lemma~\ref{lem:unglue} implies $S$ can be decomposed into the union of realizations of these pieces such that each adjacent pair of pieces intersects in the facet of each where they are glued together. 
Moreover, the projective transformations implied by the completion conditions of the connectors in Lemma~\ref{lem:cons} and transmitters in Lemma~\ref{lem:projtrans} send a projective basis among the faces of one ridge to that of another ridge, so there is a unique projective transformation between these pairs ridges.  Among these, the ridge $F_\subcube = \face(S_\subcube,f_\subcube)$ of the piece realizing $\mc{S}_\subcube$ must be projectively equivalent to $\cube$.  This determines a unique projective transformation $\phi_\subcube : \projcl(F_\subcube) \to \ralg \mb{P}^d $ such that $\phi_\subcube(F_\subcube) = \cube$, 
and $\phi_\subcube$ in turn determines a unique projective transformation 
$\phi : \projcl(F_P) \to \ralg \mb{P}^d$ by composition with the projective transformations implied by Lemmas~\ref{lem:cons} and~\ref{lem:projtrans} for the intermediate pieces $\mc{T}_{Q,P}$, $\mc{C}_Q$, $\mc{T}_{Q,F_\subcube}$. 
For each facet supporting hyperplane $h_f$, $f \in {\bf F}$ of $F_P$, each of the specified points in $\bfpentagon_h$ must have the same coordinates as the corresponding point in the corresponding facet supporting hyperplane of $P$, since it is on an edge of $\cube$, which determines $d-1$ coordinates,
and appears as the specified point of an anchor polytope $\anch(\alpha)$, which by Lemma~\ref{lem:anchor} sets the remaining coordinate to $\alpha$.  Since this determines $d$ points of $h_f$ that are not contained in any affine space of lower dimension, $\phi(h_f)$ must be a supporting hyperplane of $\face(P,f)$. Hence $F_P \projeq P$.
\end{proof}


\section{Questions}
\label{question}

The most apparent question is, what further applications does the stamp have?  
Among geometric properties of polytopes that are being studied, which of these do faces inherit, and can Theorem~\ref{thrm:log_comb=proj} be applied?  Such methods have recently been employed in \cite{adiprasitoarnau2013universality}. 

The stamp fixes a face of co-dimension 2 up to projectivity, but what about co-dimension 1.  We saw in the introduction a 3-polytope does not impose any completion condition on its facets, but such a stamp may exist in higher dimensions.  Is there a $d_0$ such that for any polytope $P$ of dimension $d \geq d_0$ there is a combinatorial polytope of dimension $d+1$ such that in all realizations a specified facet is projectively equivalent to $P$? Or, are there other properties $P$ could to satisfy to guarantee that such a combinatorial $(d{+}1)$-polytope exists? 

This article was initially motivated by the question, ``Is the Hasse diagram of the face lattice of any polytope the 1-skeleton of some other polytope?''  but we have not settled that question.  Theorem~\ref{the:bai} says that if such a polytope exists, it is not the natural candidate, the interval polytope.  

A sufficient condition was already known for a polytope to have an antiprism, that some realization is perfectly centered.  Here we saw a condition that is necessary and sufficient, that some pair of realizations is balanced, but it is not immediately clear that this new condition is actually weaker.  Does there exist a combinatorial polytope that has a pair of balanced realizations, but does not have a perfectly centered realization?

We have seen a variety of polytopes with various completion conditions.  What sort of condition can be the completion condition of a face of a polytope?  Ideally, this question would be answered by giving a formal language together with a semantic interpretation that includes polytopes among its ground types that satisfies the following.  Given a set of realizations $R$ of a combinatorial polytope $\mc{P}$, there is another combinatorial polytope $\mc{Q}$ such that $R$ is the restriction of realizations of $\mc{Q}$ to a certain face, $R = \{\face(Q,f) : \labl(Q) = \mc{Q}\}$ if and only if there exists a predicate $\psi$ in this language such that $R$ is the set where the predicate holds $R = \{ P : \psi(P) = \text{True} \}$.  Of course, it would have to be possible to formulate the completion conditions already given as a predicate in such a language.  In particular it must be possible to say that a 
polytope is fixed up to projectivity.  In the other direction, any predicate of this language would have to be projectively invariant. 

\section*{Acknowledgments}

The author would like to thank Louis Theran, Igor Rivin, G\"unter Ziegler, and Andreas Holmsen 
for helpful discussions.

\bibliographystyle{plain}
\bibliography{apip}

\end{document}